\documentclass[preprint,1p]{elsarticle} 
\usepackage[utf8]{inputenc}
\usepackage[english]{babel}
\usepackage[T1]{fontenc}
\usepackage{amsmath}
\usepackage{amsfonts}
\usepackage{amssymb}
\usepackage{amsthm}
\usepackage{amscd}
\usepackage{soul}
\usepackage{geometry}
\usepackage{enumitem} 
\usepackage{cancel} 
\usepackage{graphicx}
\usepackage{mathtools}
\usepackage{color}
\usepackage{comment}
\usepackage{mathdots}
\usepackage{algorithm}
\usepackage{algpseudocode}
\usepackage{tikz}
\usetikzlibrary{arrows}
\usepackage{color}
\usepackage{bm}
\definecolor{marin}{rgb}   {0.,   0.3,   0.7} 
\definecolor{rouge}{rgb}   {0.8,   0.,   0.} 
\definecolor{sepia}{rgb}   {0.8,   0.5,   0.} 
\definecolor{uuuuuu}{rgb}{0.26666666666666666,0.26666666666666666,0.26666666666666666}
\usepackage[colorlinks,citecolor=marin,linkcolor=rouge,
            bookmarksopen,
            bookmarksnumbered
           ]{hyperref}

\newcommand\N{\mathbb{N}}

\newcommand\Z{\mathbb{Z}}

\newcommand\R{\mathbb{R}}
\newcommand\C{\mathbb{C}}

\newcommand{\dd}{\mathrm{d}}

\newcommand{\enstq}[2]{\left\{#1~\middle|~#2\right\}}

\renewcommand{\Re}{\operatorname{Re}}
\renewcommand{\Im}{\operatorname{Im}}

\newcommand\pot{\mathrm{pot}}
\newcommand\rot{\mathrm{rot}}

\newcommand\diag{\mathrm{diag}}

\newcommand\Card{\mathrm{Card}}
\newcommand\argmin{\mathrm{argmin}}
\newcommand{\ps}{\mathrm{ps}}
\newcommand\exterior{\mathrm{ext}}
\newcommand\interior{\mathrm{int}}

\newcommand\err{\mathrm{err}}
\newcommand\GS{\mathrm{GS}}
\usepackage[normalem]{ulem}

\newtheorem{proposition}{Proposition}

\begin{document}

\begin{frontmatter}

\title{Numerical study of the Gross-Pitaevskii equation on a two-dimensional ring and vortex nucleation}

%
%
%

\author[1]{Quentin Chauleur\corref{cor1}}
\ead{Quentin.Chauleur@math.cnrs.fr}

\author[2]{Radu Chicireanu}

\author[3]{Guillaume Dujardin}

\author[2]{Jean-Claude Garreau}

\author[2]{Adam Rançon}

\cortext[cor1]{Corresponding author}

\affiliation[1]{organization={Univ. Lille, INRIA, CNRS, UMR 8523 – PhLAM – Laboratoire de
Physique des Lasers, Atomes et Molécules, F-59000},
addressline={Cité Scientifique},
postcode={59655},
city={Villeneuve-d'Ascq},
country={France}}

\affiliation[2]{organization={Univ. Lille, CNRS, UMR 8523 – PhLAM – Laboratoire de
Physique des Lasers, Atomes et Molécules, F-59000},
addressline={Cité Scientifique},
postcode={59655},
city={Villeneuve-d'Ascq},
country={France}}

\affiliation[3]{organization={Univ. Lille, INRIA, CNRS, UMR 8524 - Laboratoire Paul Painlevé F-59000},
addressline={Cité Scientifique},
postcode={59655},
city={Villeneuve-d'Ascq},
country={France}}

\begin{abstract}
We consider the Gross-Pitaevskii equation with a confining ring potential with a Gaussian profile. By introducing a rotating sinusoidal perturbation, we numerically highlight the nucleation of quantum vortices in a particular regime throughout the dynamics. Numerical computations are made via a Strang splitting time integration and a two-point flux approximation Finite Volume scheme based on a particular admissible triangulation. We also develop numerical algorithms for vortex tracking adapted to our finite volume framework.
\end{abstract}


\end{frontmatter}

\section{Introduction}

 We consider the time-dependent Gross-Pitaevskii equation 
\begin{equation} \label{gross_pitaevskii}
i \hbar \partial_t \Psi + \frac{\hbar^2}{2m} \Delta \Psi= g \mathcal{N}|\Psi|^2 \Psi + \mathcal{V} \Psi,
\end{equation}   
where $\hbar$ is Planck's constant divided by $2\pi$, $m$ is the atom mass, $g$ is the nonlinear coupling coefficient and $\mathcal{N}$ is the number of atoms. The potential $\mathcal{V}$ can refer to a stationary trapping potential (typically an harmonic or a Gaussian potential), or to an external forcing (rotation, stirring, etc.), or usually a combination of both. Finally, $\Psi$ denotes the complex macroscopic wave function.

Equation \eqref{gross_pitaevskii} stands as a fundamental model in order to describe the dynamics of Bose-Einstein condensates (BEC), a many-body quantum state of matter near absolute zero which can be described by a single wave function. This phenomenon was predicted almost a century ago, respectively by Satyendra Nath Bose \cite{bose1924} and Albert Einstein \cite{einstein1925}, and had only been recently observed in the context of atomic quantum gas experiments \cite{david1995,anderson1995}. Moreover, when set to rotation through different methods, BEC exhibit complex nonlinear phenomena, such as vortex nucleation and quantum turbulence, which have drawn more and more attention from the scientific community over the last decades~\cite{kevrekidis2004}. A quantum vortex can be defined as a topological defect where the rotational of the velocity, also called \textit{vorticity}, is non zero, whereas being equal to zero anywhere else.

From the mathematical point of view, the study of equation \eqref{gross_pitaevskii} falls into the scope of semilinear Schrödinger equation \cite{cazenave,carles2011}. These equations satisfy several conversation laws, such as the mass
\[ \mathcal{M}(t)=\int |\Psi|^2 = \mathcal{M}(0)   \]
and the following energy (when the potential $\mathcal{V}$ is a real time-independent function)
\[ \mathcal{E}(t)=\frac{\hbar^2}{2m} \int |\nabla \Psi|^2 +\frac{g\mathcal{N}}{2} \int |\Psi|^4 + \int \mathcal{V} |\Psi|^2 =\mathcal{E}(0), \]
where the first term corresponds to the kinetic energy, the second term to the energy of interaction between atoms in the mean field approximation (characterized by the constant~$g$), and the third term to the potential energy.

In this paper we are interested in the numerical simulations of the Gross-Pitaevskii equation, as well as the detection of vortex nucleation, in the particular case of a two-dimensional ring-shaped geometry motivated by experiments. In fact, the behavior of quantum fluids on ring-shaped geometry has drawn some attention in the recent years, see for instance the works \cite{ricardo2013} or \cite{ricardo2019} and references within. 

We briefly describe our numerical setting. We perform a two-order Strang Splitting integration in time as well as a two-point flux approximation Finite Volume scheme method based on a particular triangulation of our geometry. As an initial data for our dynamical simulation, we start from the ground state of equation \eqref{gross_pitaevskii}, which basically stands as the unique global minimizer of the energy under the mass constraint, and which requires a well-suited gradient descent method with adaptive step for its discrete approximation. We finally provide several post-processing algorithms: some are dedicated to the detection of vortices as well as the computation of their indices, and another one performs the decomposition of the wave function $\psi$ on the $L^2$-basis associated to the linear part of equation \eqref{gross_pitaevskii}.

Numerous simulation methods have been employed both in the physical and the mathematical literature for the computation of hte dynamics of BEC \cite{bao2013}. Without unrealistically trying to be exhaustive, we briefly list and compare a few of them in order to justify our approach:
\begin{itemize}
	\item For the space discretization, finite differences or spectral methods \cite{antoine2015,bao2013} on linear grids are commonly opposed to more elaborate Finite Elements methods on complex geometries \cite{duboscq2015,danaila2016,luddens2021}. We here perform an in-between by using a two-point flux approximation Finite Volume scheme, which is both relatively easy to implement and well-suited for the discretization of quantum fluids, inspired by the vast literature on Finite Volumes for diffusive models \cite{eymard2000}. Also note that the first author recently proved the convergence of such schemes on Delaunay-Voronoï meshes in the work \cite{chauleur2024FVGP}.
	\item  For the time integration of Gross-Pitaevskii equations, we can mention the survey \cite{besse2013}. The most popular methods are respectively based on splitting methods, Cranck-Nicholson schemes, relaxation schemes or exponential integrators. Splitting methods appear to be explicit and easy-to-implement, mass-preserving, unconditionally stable as well as memory and computationally efficient at the same time.
	\item For the computation of the ground state, there is a vast literature, starting from \cite{bao2004}, which considers \textit{normalized gradient flow} methods. These methods basically consist of a two part algorithm by first performing an imaginary time gradient descent method followed by a projection over the constrained finite dimensional manifold. They usually ensures a decreasing property of the associated discrete energy, but the only theoretical convergence result towards the ground state (even if performed in the restricted case of the cubic Schrödinger equation in one dimension starting near the ground state) seems to be a linearly implicit scheme developed in \cite{faou2018} to the best of the authors knowledge, which is the scheme we use in our simulations.
	\item For the post-processing computations around vortices, we adapt the algorithm detailed in \cite[Section 4]{dujardin2022} in the context of a two-dimensional triangulation. We also develop usual methods based on vorticity-like quantities of the quantum fluid in the finite volume framework. Let us finally mention the recent work \cite{danaila2023} where the authors use a different approach in a finite-element context.
\end{itemize}

This paper is organized as follows. In Section \ref{section_model}, by a change of variables we introduce the dimensionless Gross-Pitaevskii equation which we will use for all our ongoing computations. Section \ref{section_discretization} is devoted to both time and space discretization of this equation, as well as the introduction of a particular admissible triangulation for our Finite Volume method. In Section \ref{GS_section}, we describe the \textit{normalized gradient flow} method we will use for the numerical computation of the ground state. We develop in Section \ref{section_post_processing} the two post processing algorithms described above. We then present numerical results in Section \ref{simulations_section} which corroborate theoretical announcements as well as display vortex nucleation for a particular set of parameters. Note that all codes are available on the \textsc{Gitlab} page \url{https://plmlab.math.cnrs.fr/chauleur/codes/-/tree/main/bose_einstein_condensates_ring}. 

\begin{figure}[h]
	\centering	
	\includegraphics[width=1\textwidth,trim = 0cm 0cm 0cm 0cm, clip]{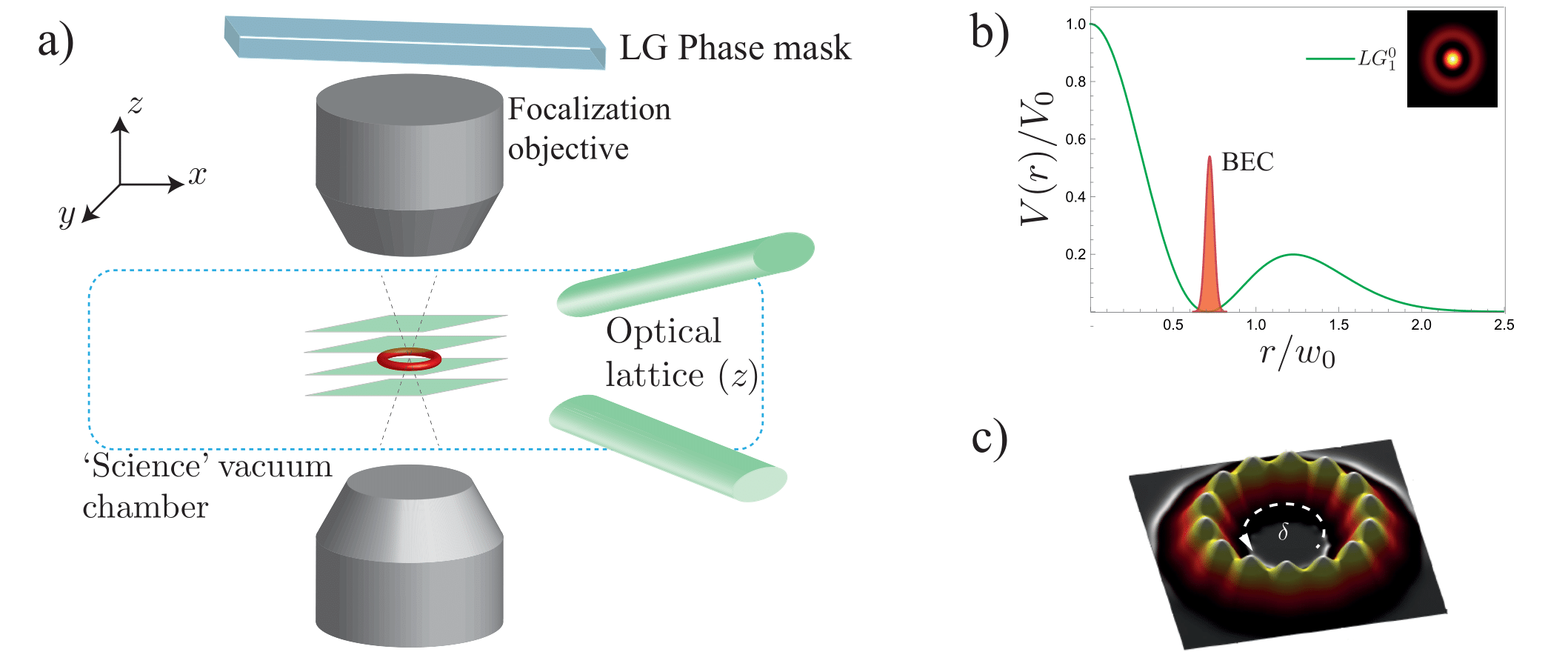}	
    \caption{Atomic BEC 2D ring trap configuration. a) Schematic representation of the setup. The LG laser beams (not shown) are focused down to the BEC (red) through high numeric-aperture optics. In the vertical direction, the BEC is trapped in a single node of an additional optical lattice, which `freezes' the motion of atoms along $z$, effectively reducing the system's dimensionality. b) Radial profile of the ring trapping potential, proportional to the intensity of the LG mode $L_1^0$ (inset). c) Rotating potential modulation used to stir the BEC into motion.}
\label{fig:cold_atoms_setup.png}
\end{figure}

\section{The Gross-Pitaevskii model} \label{section_model}

\subsection{Physical motivation}
The main aim of this work is to perform accurate numerical simulations which could be reproduced for quantum gas experimental setting at the PhLAM laboratory located at Villeneuve d'Ascq, schematically described in Figure~\ref{fig:cold_atoms_setup.png}.a). The quantum gas is expected to form and to evolve as a ring-shaped structure. A first important assumption, motivated by the experimental setup, is that the dynamics of the Gross-Pitaevskii model in our setting will be confined to a region where the height, namely the third coordinate $z$, will be irrelevant, which can be experimentally achieved through tight confinement of the atoms along this direction.

For the experimental setup, we consider a BEC of potassium atoms (isotope $^{39}$K), which allow to tune the strength of the atom-atom interactions by applying a external magnetic field (using the so-called `Feshbach resonances'~\cite{Chin2010}). In such setups the gas can be confined and manipulated using `optical dipole potentials', created by off-resonant laser light, where the potential is proportional to the local light intensity. The intensity profile of the laser beams can be conveniently shaped to create different potential landscapes, and even to modify the `effective dimension' of the gas. Here, we will consider (\textit{i}) a ring-shaped trapping potential, which defines the geometry and dimension (2D) of the problem and (\textit{ii}) a rotating periodic ring modulation, which will be used to set the trapped BEC into motion.

The specific values of the physical constants appearing throughout this paper will be of a particular importance, related to the experiment, and their physical relevance in a practical context will be systematically pointed out and discussed. Note that a full parametric study on the influence of these constants for the nucleation of vortices and quantum turbulence is beyond the scope of this paper. An experimental study of this particular model in order to corroborate our theoretical results is also a task that we plan to tackle in the future. 

\subsubsection{Stationary 2D trapping ring potential}

The 2D ring potential can be created using Laguerre-Gaussian (LG) laser beams~\cite{Allen1992}, which can be realized in practice using phase masks~\cite{Allen1992,Rhodes2006}. For a beam propagating along $z$, the transverse ($x_1 x_2$) electric field profile near the focal point is: $\mathcal{E}_{\ell,p}(r,\theta) \propto L_p^{|\ell |}(2 r^2/w_0^2) \exp(-i\ell\theta)$. Here, $r$ and $\theta$ are the polar coordinates in the $x_1 x_2$ plane, $w_0$ is the waist of the beam (typically a few microns), and the indices $p\in\mathbb N$ and $\ell\in\mathbb Z$ of the generalized Laguerre polynomial $L_p^\ell$ can be chosen to achieve the desired $x_1 x_2$ trapping potential. The beam intensity ($I\propto |\mathcal{E}_{\ell,p}|^2$) has cylindrical symmetry with respect to the $z=0$ axis, and presents one or several radial extrema, according to $\ell$ and $p$. For this work, the $L_1^0$ mode appears promising, as its radial profile presents two light rings, with an intermediate `dark' region, see Figure~\ref{fig:cold_atoms_setup.png}.b), where atoms will be trapped using a blue-detuned laser. The ring radius, corresponding to the potential minimum, is $r_0=w_0/\sqrt{2}$.

The confinement along the LG beam’s propagation axis ($z$) is very loose, as the LG beam intensity near the focal point (where the BEC is trapped) is almost uniform along $z$. Trapping along this longitudinal direction can be achieved by adding a tightly-confining ``optical lattice'', i.e. a sinusoidal potential obtained with an additional optical standing wave. The BEC can be loaded in the ground state of a single site of this lattice~\cite{ville2017}. Provided all relevant energy scales considered in this study are lower than the difference between the ground and the first excited state of the lattice site (in the harmonic approximation), the dynamics along $z$ will be effectively \textit{frozen}, and this physical situation can be treated as an `effective' 2D problem~\cite{ville2017}. This situation will be assumed throughout the rest of the paper.

Note that strong directional confinement is a common technique in atomic BEC experiments to reduce the effective dimensionality of inherently physically-3D systems~\cite{Gorlitz2001,Hadzibabic2004,Kohl2005}. From a numerical perspective, this allows us to only work on a two-dimensional set $\Omega \subset \R^2$. Although obviously not perfectly fitting the reality, this assumption appears to be pretty effective from the computational cost perspective.

\subsubsection{Rotating potential modulation}

To study vortex nucleation, one needs in practice to `stir' the 2D ring-trapped BEC~\cite{Wright2013,Amico2021}. The method which we propose in this paper is to use a rotating azimuthal potential modulation (along $\varphi$), which can be achieved using additional LG laser beams. LG modes with $\ell\neq 0$ are known to carry an `orbital angular momentum' that can be transferred to the atoms. By superposing two counter-propagating LG modes with opposite azimuthal $\ell$-numbers, $L_p^\ell$ and $L_p^{-\ell}$~\cite{Ngcobo2013}, which have azimuthal phase dependence $\exp(\pm i \ell \theta)$. This creates a spatial modulation $\propto \sin(2\ell \theta)$ along the ring, see Figure~\ref{fig:cold_atoms_setup.png}.c). Moreover, by slightly changing the optical frequencies of the beams ($\omega$ and $\omega+\delta$), one obtains a rotating sinusoidal modulation $\propto \sin(2\ell\theta+\delta t)$, whose rotation velocity can be easily controlled experimentally via the parameter $\delta$. The modulation thus acts as a rotating periodic potential for the BEC, with a spatial period $\pi r_0/\ell$ which can be varied by changing $\ell$ (typically between 1 and 10~\cite{Ngcobo2013}) or the 1D ring's radius $r_0$.

\subsection{The dimensionless model} \label{model_section}
In order to computationally solve equations arising from quantum mechanics, it is common to work with a dimensionless equivalent of equation \eqref{gross_pitaevskii} so we do not have to worry about the effects of extremely small constants (for instance $\hbar \simeq 10^{-34} \text{J} \cdot\text{s}$) or huge ones (for instance $\mathcal{N} \simeq 10^4$ atoms) on our numerical simulations. To this end, we first need to describe the stationary potential $\mathcal{V}$, which is assumed to be a radial Gaussian potential of the form
\[ \mathcal{V}(r)= -\mathcal{V}_0 \exp \left(-\frac{(r-r_0)^2}{\Delta r^2} \right),  \]
where $r=\sqrt{x_1^2+x_2^2}$, $x=(x_1,x_2) \in \R^2$. Here the positive constants $\mathcal{V}_0$, $r_0$ and $\Delta r$ respectively denote the amplitude, the radius of trap and typical width of ring of the Gaussian potential $\mathcal{V}$, and they should be determined by the power and the geometry of the laser beam during the experiment, with typical values $\mathcal{V}_0 \simeq 3.14 \text{MHz}\times \hbar$.

If we denote by $n(r)=\mathcal{N}|\Psi|^2$ the atom local density, we have
\[ \mathcal{N}=\int_{\R^2} n(r) \dd r = \mathcal{N} \int_{\R^2}  |\Psi|^2 \dd r, \]
which implies the usual normalization $\int |\Psi|^2 \dd r = 1$. In two dimensions, $|\Psi|^2$ then has the dimension of the inverse of a surface, and as $g \mathcal{N} |\Psi|^2$ stands for the interaction energy, $g$ has the dimension of an energy times a surface. Hence the nonlinear coupling coefficient $g$ is then related to a unique atomic magnitude, the s-wave scattering length $a$ \cite{petrov2001}, through the relation
\[ g=\frac{\sqrt{8\pi} a}{ l_z} \frac{\hbar^2}{m},  \]
with $l_z$ the typical length of the gaz in the direction perpendicular to the plane. For example for potassium, we have $a \simeq 5$~nm and $m=6.5 \times 10^{-26}$~kg, as well as $l_z \simeq 0.5$ \textmu m, which gives a typical dimensionless value of $mg/\hbar^2 \simeq 0.05$ for 2D quantum gases.
We are now going to write the Gross-Pitaevskii equation \eqref{gross_pitaevskii} in terms of dimensionless quantities (pointed out by a $'$), choosing for  the energy unite $\hbar \omega_c =\frac{\hbar^2}{2m \Delta r^2}$, and for the length unite $r_0$, so $r'=r/r_0$. This implies the normalization condition
\[ 1= \int_{\R^2}|\Psi|^2 \dd r = \int_{\R^2} r_0^2 |\Psi|^2 \dd \left( \frac{r}{r_0^2} \right)=\int_{\R^2} |\psi|^2 \dd r'   \]
so $|\psi|^2=r_0^2|\Psi|^2$. Furthermore, dividing equation~\eqref{gross_pitaevskii} by $\hbar \omega_c$, we get
\[i\frac{\partial \psi}{\partial ( \omega_c t)}+\frac{\hbar}{2m \omega_c r_0^2} \Delta_{x'} \psi  = \frac{g\mathcal{N}}{\hbar \omega_c r_0^2}|\psi|^2 \psi + \frac{\mathcal{V}}{\hbar \omega_c} \psi,   \]
which can be written under the form
\[ i \frac{\partial \psi}{\partial t'} + \frac{1}{2m'} \Delta_{x'} \psi= \gamma |\psi|^2 \psi + V \psi,  \]
with effective variables
\[ t'=\omega_c t, \quad m'= \frac{r_0^2}{2 \Delta r^2}, \quad V=\frac{\mathcal{V}}{\hbar \omega_c} \quad \text{and} \quad \gamma=\frac{g\mathcal{N}}{\hbar \omega_c r_0^2}. \]
In particular the trapping potential wan be rewrite in dimensionless unite by
\[ V(r)=-\frac{\mathcal{V}_0}{\hbar\omega} \exp \left(-\frac{(r-r_0)^2}{\Delta r^2} \right) =-V_0 \exp \left(-2m' (r'-1)^2 \right).   \]
Dropping the $'$ from now on, we are thus brought back to the study of the dimensionless Gross-Pitaevskii equation:
\begin{equation} \label{GP} 
i \partial_t \psi + \frac{1}{2m} \Delta \psi= \gamma |\psi|^2 \psi + V\psi,
\end{equation}
with an initial condition $\psi(0,\cdot)=\psi_0$, and a positive effective mass $m>0$ and nonlinear interaction constant $\gamma >0$. In this setting, the potential $V=V(t,x)$ is composed of a time-independent trapping potential $V_{\pot}$ and a time rotating potential $V_{\rot}$, namely
\[  V= V_{\pot} +  V_{\rot},  \]
with 
\[ V_{\pot}(r)=-V_0 \exp \left( -2m(r-1)^2 \right), \quad V_0 >0,  \]
and 
\begin{equation} \label{forcing_potential_definition}
V_{\rot}= V_p V_{\pot}(r) \sin(n_{\theta}\theta-\omega t))
\end{equation}  
where $ 0 \leq V_p \leq 1$, $n_{\theta} \in \N^*$ and $\omega \in \R$. As explained before, this form of potential can be obtained by making an interference pattern of two LG laser beams with opposite azimuthal dependences $\exp(\pm i \ell \theta)$. We finally emphasize that at first order near $r \simeq 1$ the Gaussian potential $V_{\pot}$ corresponds to a harmonic potential with radial symmetry of the form
\[ V_{\pot}(r) \simeq   V_0 (2 m(r-1)^2-1).  \]

\subsection{Assumption on the geometry}
We now wish to simulate the dynamics of the condensate on a ring-shaped domain of the form :
\begin{equation*} \label{ring_geometry}
 \Omega := \enstq{x \in \R^2}{ r_{\min} \leq |x| \leq r_{\max}  }, 
 \end{equation*}
with $0<r_{\min}<1<r_{\max}$. Furthermore, in order to solve equation \eqref{GP}, we also impose some Dirichlet boundary conditions to equation \eqref{GP}, namely
\begin{equation} \label{Dirichlet_conditions}
	\psi(\cdot,x)=0, \quad \forall x \in \partial \Omega.
\end{equation} 
Note that from a physics experiment point of view, particles hitting the edge of $\Omega$ with outwards velocity leave the domain $\Omega$. This phenomena may happen in this setting because we are providing the system with energy {\it via} the time-dependant potential $V_{\rot}$. This phenomena is not transcribed into our setting, since Dirichlet conditions \eqref{Dirichlet_conditions} imply some artificial reflection at the edge of the domain $\Omega$. This is an important task that we have to keep in mind during the simulations of Section \ref{simulations_section}.

\section{Discretization of the problem} \label{section_discretization}

\subsection{Time discretization by Strang splitting} \label{splitting_section}
Let $\tau >0$ be the time discretization step size. We will perform a semi-discretization in time with a Strang splitting method in order to compute the dynamic of equation~\eqref{GP} with initial condition $\psi(0,x)=\psi_0(x)$,~$x\in \Omega$. The operator splitting methods we consider for the time integration of \eqref{GP} are based on the following splitting
 \[  \partial_t \psi = A( \psi ) + B ( \psi ),   \]
where
\[ A( \psi) = \frac{1}{2m} i \Delta \psi, \quad B ( \psi ) =-i \gamma |\psi|^2 \psi  -i V \psi  \]
and the solutions of the following subproblems
\[ \left\{ \begin{aligned}
 &  i \partial_t u(t,x)=- \frac{1}{2m}  \Delta u(t,x), \quad &  u(0,x)=u_0(x),  \\    
&  i  \partial_t w(t,x)= \gamma |w(t,x)|^2 w(t,x) + V(t,x) w(t,x), \quad & w(0,x)=w_0(x),  
\end{aligned} \right. \]
 for all $x \in \Omega$ and $t>0$. The associated operators are explicitly given, for $t \geq 0$, 
\[ \left\{ \begin{aligned}
 & u(t+\tau,\cdot)=\Phi^{\tau}_A (u(t,\cdot)) = e^{i\frac{\tau}{2m}  \Delta} u(t,\cdot), \\
& w(t+\tau,\cdot)=\Phi^{\tau, t}_B (w(t,\cdot)) =  e^{- i \int_{0}^{\tau} V(t+s,\cdot) \dd s-i \tau \gamma |w(t,\cdot)|^2} w(t,\cdot).
\end{aligned} \right. \]
In particular, from the expressions of $V_{\pot}$ and $V_{\rot}$ given in Section \ref{model_section}, we have
\[ \Phi^{\tau,t}_B (v(t,\cdot)) = \exp \left(- i V(r) \left[ \frac{\tau}{2} + \frac{1}{\omega} \left(\cos(n_{\theta} \theta - \omega (t+\tau))- \cos(n_{\theta} \theta - \omega t ) \right) \right] \right).\]
The Strang splitting is then defined by the composition of flows
\begin{equation} \label{definition_strang_splitting}
  \psi(t+\tau,\cdot) = \Phi^{\tau,t}(\psi(t,\cdot)) :=  \Phi^{\frac{\tau}{2},t+\frac{\tau}{2}}_B \circ \Phi^{\tau}_A \circ \Phi^{\frac{\tau}{2},t}_B \left( \psi(t,\cdot)  \right), 
 \end{equation}
which provides a time-integration method of order 2 in $\mathcal{O}(\tau^2)$ (see for instance \cite{besse2002}). Note that we only evaluate the exponential of the Laplace operator once per iteration, as $\Phi_A$ will be the most expensive flow to compute (compared to the direct integration of the flow $\Phi_B$) by the method of finite volume that we are going to present. Also note that 
\[ \psi(t+2\tau,\cdot) = \Phi^{\frac{\tau}{2},t+\frac{3\tau}{2}}_B \circ \Phi^{\tau}_A \circ \underbrace{\Phi^{\frac{\tau}{2},t+\tau}_B \circ \Phi^{\frac{\tau}{2},t+\frac{\tau}{2}}_B}_{\Phi_B^{\tau,t+\frac{\tau}{2}}} \circ \ \Phi^{\tau}_A \circ \Phi^{\frac{\tau}{2},t}_B \left( \psi(t,\cdot)  \right) \]
so one can compute only once $ \Phi^{\tau,t}_B$ (in a similar way to the one-order Lie splitting) at each step for computational effectiveness.

\subsection{Space discretization by Finite Volumes}  \label{finite_volume_section}
In order to discretize the space variable $x \in \Omega$ in equation \eqref{GP}, we will perform a standard two-point flux approximation finite volume scheme of order 1 (see for instance~\cite{eymard2000} for a classical introduction to this method), based on a regular triangulation $\mathcal{T}$ of our domain $\Omega$. In particular, in this part we will assume the existence of such triangulation~$\mathcal{T}$, and we refer to Section \ref{triangulation_section} for further discussions about the triangulation we will adopt in our simulations.

As the Laplace operator $\Delta$ is the only component of equation \eqref{GP} involving space derivatives, we briefly recall how finite volume schemes actually discretize the following Laplace equation on $\Omega$ with Dirichlet conditions:
\begin{equation} \label{laplace_eq}
  \left\{
      \begin{aligned}
        -\Delta u(x) =f(x), &\quad x \in \Omega,\\
        u(x)=0, &\quad x \in \partial \Omega,
      \end{aligned}
    \right.
\end{equation}
where $f:\Omega \rightarrow \C$ denotes a regular function. Let $\mathcal{T}$ be a finite family of open triangular disjoints subsets of a polygonal approximation of $\Omega$, such that two triangles having a common edge have also two common vertices. Also assume that all angles of the triangles are less than $\pi/2$. This define an admissible mesh in the sense that there exists a family of points $(x_K)_{K \in \mathcal{T}}$, such that for any $K,L \in \mathcal{T}$ sharing a common edge denoted $K|L$, the points $x_K \in K $ and $x_L \in L$ satisfy the orthogonality condition
\[ \left[x_K,x_L \right] \perp K|L. \]
Integrating equation \eqref{laplace_eq} over a triangle $K \in \mathcal{T}$, we obtain
\[ |K| f_{K} := \int_{K} f = - \Delta \int_{K} u = \sum_{\sigma \in \mathcal{E}_{K}} \left( - \int_{\sigma} \nabla u \cdot \nu_{K, \sigma}  \right),  \]
where $\mathcal{E}_{K}$ denotes the set of the edges of the triangle $K$, and $\nu_{K, \sigma}$ the unit normal to the edge $\sigma$, outward to $K$. We also denote $\mathcal{E}$ the set of the edges defined by the triangular mesh $\mathcal{T}$, $\mathcal{E}_{\exterior} \subset \mathcal{E} $ the set of the edges located on the boundary of $\tilde{\Omega}$, and $\mathcal{E}_{\interior} = \mathcal{E} \backslash  \mathcal{E}_{\exterior}$. We denote $L \in \mathcal{T}$ the triangle such that $\sigma = K | L$, then we can approximate
\[ - \int_{\sigma} \nabla u \cdot \nu_{K, \sigma} \simeq - |\sigma| \frac{u_{L} - u_{K}}{d_{K,L}}  \]
if $\sigma \in \mathcal{E}_{\interior}$, where $d_{K,L}$ denotes the distance between the circumcenter $x_{K}$ of the triangle $K$ and the circumcenter $x_{L}$ of the triangle $L$, or 
\[ - \int_{\sigma} \nabla u \cdot \nu_{K, \sigma} \simeq  - |\sigma| \frac{ - u_{K}}{d_{K,\sigma}} \]
if $\sigma \in \mathcal{E}_{\exterior}$, as we impose some Dirichlet boundary conditions, and where $d_{K,\sigma}$ denotes the distance between the circumcenter $x_{K}$ of the triangle $K$ and the edge $\sigma$. Denoting $N=\Card(\mathcal{T})$ the number of triangles, we then look for a solution $U=\left(U_{K} \right)_{K \in \mathcal{T}} \in \R^N$ such that
\begin{equation*} 
  \left\{
      \begin{aligned}
        & \sum_{\sigma \in \mathcal{E}_{K}} \mathcal{D}_{K,\sigma} \left( U \right)= |K| f_{K}, &\quad \forall K \in \mathcal{T},\\
        & \mathcal{D}_{K,\sigma} \left( U \right)= |\sigma| \frac{ U_L-U_K}{d_{K,L}}  , &\quad \text{if} \ \sigma=K | L \in \mathcal{E}_{\interior},\\
        & \mathcal{D}_{K,\sigma} \left( U \right)=  - |\sigma| \frac{U_{K}}{d_{K,\sigma}} , &\quad \text{if} \ \sigma \in \mathcal{E}_{\exterior},
      \end{aligned}
    \right.
\end{equation*}
which define a system of $N$ linear equations with $N$ unknowns, which can be put into the matrix form
\[  A U = F,  \]
where $F=\left(|K| f_{K} \right)_{K \in \mathcal{T}} \in \R^N$. Note that the matrix $A$ is the discrete counterpart of the Laplace operator $\Delta$, and that by construction, the matrix $A$ is symmetric. Moreover, the matrix 
\[ A_{\mathcal{T}} := \left(\frac{1}{|K_i|} A_{i,j} \right)_{1\leq i,j \leq N} \]
is symmetric in the base of our triangulation $\mathcal{T}$, which means that defining the scalar product 
\begin{equation} \label{definition_scalar_product_triangulation}
\langle U,V \rangle_{\mathcal{T}} := \Re \left( \sum_{K \in \mathcal{T}} U_{K} \overline{V_{K}} |K| \right), \quad U,V \in \R^N, 
\end{equation} 
with respect to the triangulation $\mathcal{T}$, then
\[  \langle A_{\mathcal{T}}U,V \rangle_{\mathcal{T}} =  \langle U,A_{\mathcal{T}}V \rangle_{\mathcal{T}}.  \]

\subsection{Triangulation of the geometry} \label{triangulation_section}
We now want to generate an admissible mesh, with all angles less than $\pi/2$. One could rely on triangulations generated by mesh-generation software such as \textsc{GMSH} \cite{gmsh} with the "Delaunay" option, see Figure \ref{fig:gmsh_triangulation_15222_triangles.png}.

 \begin{figure}[h]
	\centering
		\includegraphics[width=0.8\textwidth,trim = 0cm 0cm 0cm 0cm, clip]{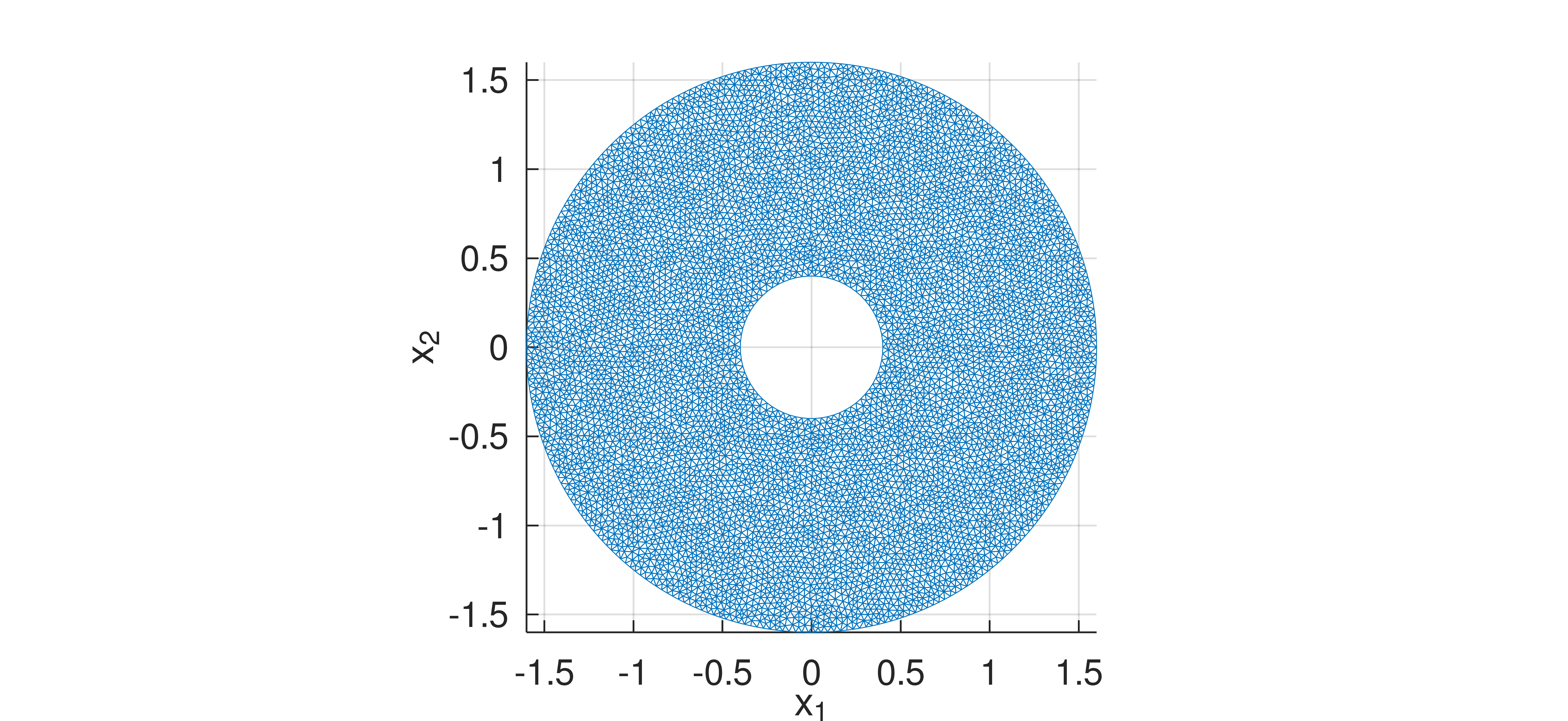}	
		\caption{GMSH triangulation of 15222 triangles with $r_{\min}=0.4$, $r_{\max}=1.6$ and \textit{element size factor} $h=0.08$.}
	\label{fig:gmsh_triangulation_15222_triangles.png}
\end{figure}

However, we will rather generate our own mesh taking into account two main features coming from our particular setting:
\begin{itemize}
\item We attend to observe most part of our dynamic at the "center" of the ring $\Omega$, which means the zone near the unit circle $\mathbb{S}^1$, so we would rather seek for non-uniform mesh grid. In fact, as we are looking for vortex points at the center part of the ring, which might be very localized zeros of the wave function $\psi$, we would like some mesh refinements at this specific area to detect more easily this phenomena. At the contrary, with Dirichlet conditions and a strong trapping potential $V_{\pot}$ (in the sense that $V_0 \gg 1$), we expect no dynamics at the edge of the ring $\Omega$, hence we would be satisfied by rough mesh size around $\partial \Omega$.
\item Taking $V_{\rot}=0$, we easily see that any solution of equation \eqref{GP} with radial initial condition $\psi_0(x)=\psi_0(r)$ inherits this radial symmetry property for any time $t \geq 0$, a feature that will would like to preserve as much as possible for the induced discrete scheme. Unfortunately, non-symmetric meshes instantly break this property, hence we will rather build a triangulation presenting some radial symmetry.
\end{itemize}
We now describe how we construct the symmetric concentrated mesh that we will use for our simulations in Section \ref{simulations_section}, providing the values of $r_{\min}$ and $r_{\max}$ of the ring $\Omega$, as well as the control of stepsize of the triangles $h>0$.

\subsubsection{Building the symmetric mesh}
We will first discretize the unit circle $\mathbb{S}^1$ by generating $N_p$ points $(x_k)_{1\leq k\leq N_p}$ ($N_p$ for "number of points") such that
\[ X_k=\left( \cos \left( \frac{2(k-1)\pi}{N_p} \right), \sin\left( \frac{2(k-1)\pi}{N_p} \right)   \right).   \]
We then perform $N_c$ dilatation of the vector $(X_k)_k$ with equidistant radius $(r_j)_{1\leq j \leq N_c}$ ($N_c$ for "number of circles") such that $r_1=r_{\min}$, $r_{N_p}=r_{\max}$ and for all $1\leq j \leq N_c-1$, 
\[ r_1 \leq \ldots \leq r_{N_c} \quad \text{and} \quad r_{j+1} - r_j = \frac{r_{\max}-r_{\min}}{N_c}.  \]
We denote by 
\[ x_k^j= \left( r_j \cos \left( \frac{2(k-1)\pi}{N_p} \right), r_j \sin\left( \frac{2(k-1)\pi}{N_p} \right)   \right)    \]
this collection of points.
For each point $x_k^j$ with $1 \leq k \leq N_p-1$ and $1 \leq j \leq N_c-1$, we can then define two triangles 
\[  t_{k,j}^1= \left( x_k^j, x_k^{j+1}, x_{k+1}^{j+1} \right) \quad \text{and} \quad  t_{k,j}^2= \left( x_k^j, x_{k+1}^j, x_{k+1}^{j+1} \right).  \]
The same way, you can define two triangles for $x_{N_p}^j$, namely,
\[  t_{N_p,j}^1= \left( x_{N_p}^j, x_{N_p}^{j+1}, x_1^{j+1} \right) \quad \text{and} \quad  t_{N_p,j}^2= \left( x_{N_p}^j, x_1^j, x_1^{j+1} \right).  \]
Unfortunately, triangles defined this way rapidly tends to have angles higher than $\pi/2$ as $N_p \rightarrow \infty$. In order to circumvent this problem, we define a little rotation of $-(j-1)\pi/N_p$ for the $(x_k^j)_k$ points with $1 \leq j \leq N_c$, which gives the new definition
\[ \tilde{x}_k^j= \left( r_j \cos \left( \frac{2(k-1)\pi}{N_p} -\frac{(j-1)\pi}{N_p} \right), r_j \sin\left( \frac{2(k-1)\pi}{N_p} -\frac{(j-1)\pi}{N_p} \right)   \right).    \]
We can then define the triangulation $\mathcal{T}=(\tilde{t}_{k,j}^l)_{1\leq k \leq N_p, 1 \leq j \leq N_c, l=1,2}$. 

\subsubsection{Concentrated symmetric mesh}

As we expect that a large part of the dynamics will be located at the middle portion of the ring $\Omega$ (roughly at $r \simeq 1$), we would like to build some refine mesh near this area. We now present how to choose the number of points per circle $N_p$ and the number of circles $N_c$ in order to have a concentrated mesh which satisfy the strict Delaunay condition. As we embed $N_c$ circles into $\Omega$, we look for a quadratic function
\[ f(x)=\alpha \left(x-\frac12 \right)^2+ \frac{r_{\max}-r_{\min}}{2N_c}, \quad 0\leq x \leq 1,   \]
with $\alpha$ to be fixed, such that for the radius $(r_k)_{1 \leq k \leq N_c}$ of these circles, we have $r_{k+1}-r_k=f(k/N_c)$ and the endpoints conditions $r_1=r_{\min}$ and $r_{N_c}=r_{\max}$. Here note that we impose the arbitrary condition that $r_{k+1}-r_k$ is bounded by below by $\frac{r_{\max}-r_{\min}}{2N_c}$, which corresponds to the minimal thickness of the embedded ring. From the condition
\[ r_{\max}-r_{\min}=\sum_{k=1}^{N_c} (r_{k+1}-r_k) = \alpha \sum_{k=1}^{N_c} \left(\frac{k}{N_c}-\frac12 \right)^2+ \frac{1}{2}  \]
we infer that
\[ \alpha= \frac{6 N_c}{N_c^2+2}(r_{\max}-r_{\min}).   \]
We now look for a minimum of points $N_p$ for every circles in order to ensure that every angle of our triangulation is strictly smaller than $\pi/2$.
\begin{proposition} \label{proposition_admissible_triangulation}
We fix a stepsize $h>0$. Let $N_c=\left\lceil\frac{r_{\max}-r_{\min}}{h} \right\rceil$ and 
\[ N_p = \left\lceil 2\pi\sqrt{2} \frac{N_c}{1-\frac{r_{\min}}{r_{\max}}} \right\rceil.   \]
Then the triangulation $\mathcal{T}$ satisfy the strict Delaunay condition.
\end{proposition}
\begin{proof}
Let's first note that by construction, all the triangles $L$ of our triangulation $\mathcal{T}$ are isosceles, so the two equal angles are necessary lower than $\pi/2$. We denote $\theta_{L}$ the remaining angle, $a_{L}$ the two equal sides and $b_{L}$ the remaining side (or $\theta$, $a$ and $b$ when there is no risk of confusion). We now have to split our analysis between triangles pointing towards the point $(0,0)$ and triangles pointing towards the exterior of the domain $\Omega$. \\
We begin with triangles pointing towards the center of the ring. By Alkashi formula, we have
\[ b^2=2a^2(1-\cos \theta).   \]
As we want $\theta < \pi/2$ we need $\cos \theta >0$, so
\[ \cos \theta = 1 -\frac{b^2}{2a^2} \Rightarrow \frac{b^2}{2} < a^2,   \]
hence
\[ \theta < \pi/2 \Leftrightarrow \frac{a}{b} >\frac{1}{\sqrt{2}}.  \]
Hence we need a lower bound of $a$ and an upper bound of $b$. By classical geometry, we have the rough bounds 
\[a> r_{k+1}-r_k \quad \text{and} \quad b < \frac{2\pi r_{k+1}}{N_p}. \]
Hence we want that for all $1 \leq k \leq N_c$, 
\[ \frac{r_{k+1}-r_k}{\pi r_{k+1}} \geq \sqrt{2}   \]
by simplification. Reversing the inequality and taking the supremum over $k$, we need
\[ N_p \geq \pi \sqrt{2} \frac{\max_k(r_k)}{\min_k(r_{k+1}-r_k)}.   \]
We obviously have $\max_k(r_k)=r_{\max}$. On the other hand, we know that $r_{k+1}-r_k=f(k/N_p)$, so
\[ \min_k f(k/N_c)= (r_{\max}-r_{\min}) \min_k \left[ \frac{6N_c}{N_c^2+2} \left(\frac{k}{N_c} - \frac12 \right)+\frac{1}{2N_c}   \right] \geq \frac{r_{\max}-r_{\min}}{2N_c},   \]
which is achieved for $k=N_c/2$ if $N_c$ is even (and actually gives a lower bound if $N_c$ is odd). Hence we impose
\[ N_p \geq 2\pi \sqrt{2} \frac{r_{\max} N_c}{r_{\max}-r_{\min}}.   \]
In the case of triangles oriented outside the ring, we are going to use that
\[ \tan(\theta/2) = \frac{b}{2c},   \]
where $c$ denotes the height of the isosceles triangle, so we need an upper bound of $b$ (that we already compute) and a lower bound of $c$. An elementary geometric argument shows that in this case $ c > r_{k+1}- r_k $, hence we are back to the condition
\[  N_p> \pi \frac{\max_k (r_k)}{\min_k ( r_{k+1}-r_k)},  \]
which is less constraining than the previous one, hence the proposition is proved.
\end{proof}

\begin{proposition}
Let $\mathcal{T}$ be the triangulation defined by the above procedure, with $r_{\min}$, $r_{\max}$ and $h>0$ given. Then, following notations from Section \ref{finite_volume_section}, we have that for all edges $\sigma \in \mathcal{E}$, 
\[   |\sigma|=\mathcal{O}(h) \quad \text{as} \quad h \rightarrow 0 .\]
Moreover, for all triangles $K \in \mathcal{T}$, 
\[   |K|=\mathcal{O}(h^2) \quad \text{as} \quad h \rightarrow 0 .\]
\end{proposition}
\begin{proof}
Going back to the notation of the proof of Proposition \ref{proposition_admissible_triangulation}, we first consider the case of an isosceles triangle $K$ (with two equal sides of length $a$ and one side of length $b$ of height length $c$) pointing out to the point $(0,0)$ and defined between radius $r_k$ and $r_{k+1}$. Let's first note that by definition of $N_p$,
\[ b \leq \frac{2\pi r_{k+1}}{N_p} \leq \frac{2\pi r_{\max}}{N_p} \leq \frac{h}{\sqrt{2}}.   \]
On the other hand, by Pythagore's formula,
\[ a^2 =c^2 +\frac{b^2}{4} \leq (r_{k+1}-r_k)^2 +\frac{h^2}{8},   \]
so as $r_{k+1}-r_k=f(k/N_c)$, we infer by definition of $N_c$ that
\begin{equation*}
\begin{aligned}
 r_{k+1}-r_k &  \leq  \frac{6 N_c}{N_c^2+2} (r_{\max}-r_{\min}) \left( \frac{k}{N_c} -\frac12 \right)^2 +  \frac{r_{\max}-r_{\min}}{2N_c} \\
 & \leq  \frac{r_{\max}-r_{\min}}{2} \left( \frac{3N_c}{N_c^2+2} + \frac{1}{N_c}  \right) \leq \frac{2(r_{\max}-r_{\min})}{N_c} \\
 &  \leq 2h,
 \end{aligned}
 \end{equation*}
 so we well get that $a$ and $b$ are $\mathcal{O}(h)$ as $h \rightarrow 0$. In order to deal with triangles pointing out outwards, we just use the fact that $c \leq 2(r_{k+1}-r_k)$ by elementary geometry, which gives the same result. The result on the area $|K|$ of the triangle $K$ naturally follows.
\end{proof}

We illustrate such a concentrated symmetric mesh in Figure \ref{fig:concentrated_triangulation_14850_triangles.png}. In the following, we will denote by $N$ the number of triangles of the triangulation $\mathcal{T}$ generated through the above procedure, given a stepsize $h>0$.

 \begin{figure}[h]
	\centering
		\includegraphics[width=0.8\textwidth,trim = 0cm 0cm 0cm 0cm, clip]{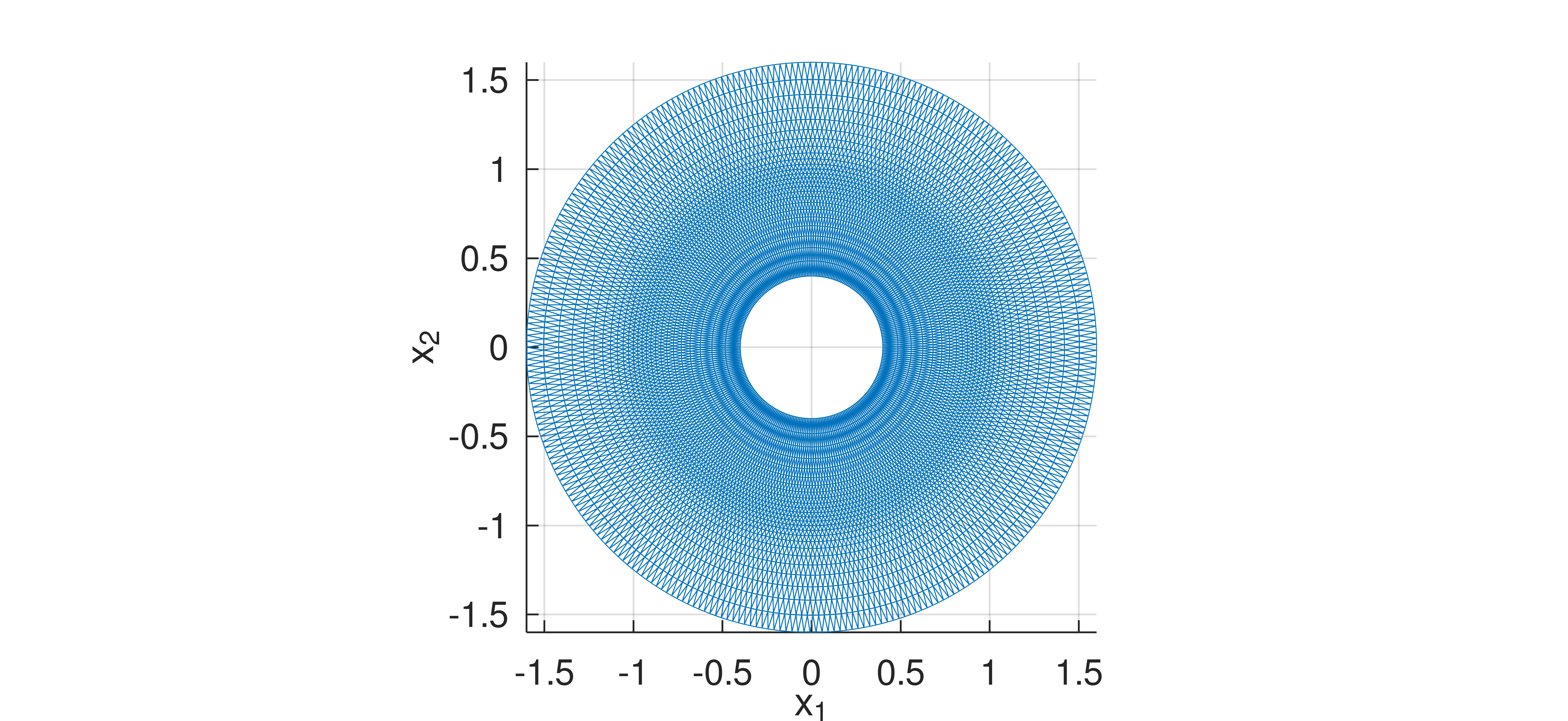}	
		\caption{Concentrated symmetric triangulation of $N=14850$ triangles with $r_{\min}=0.4$, $r_{\max}=1.6$, \textit{step size} $h=0.05$, a number of circles $N_c=25$ and a number of points per circle $N_p=297$. }
	\label{fig:concentrated_triangulation_14850_triangles.png}
\end{figure}

\section{Computation of the ground state} \label{GS_section}
We now aim at computing the ground state of equation \eqref{GP}, which is defined as the unique minimizer $\phi \in L^2(\Omega)$ of the energy
 \begin{equation*} \label{energy}  
E(\psi)= \int_{\Omega} \left(\frac{1}{2m} \left| \nabla \psi(t,x)  \right|^2 + V_{\pot}(x) |\psi(t,x)|^2 + \frac{\gamma}{2} |\psi(t,x)|^4 \right) \dd x 
 \end{equation*} 
of the Gross-Pitaevskii equation \eqref{GP} under the mass constraint
 \[M(\psi)= \int_{\Omega} |\psi(t,x)|^2 \dd x =1 .\]
 This ground state solution (or more precisely its discrete approximation) will then serve as our initial state function $\psi_0(x)=\phi(x)$ in our simulations. Hence, in the case without rotation ($\omega=0$), the simulation of the dynamics should remain near the initial state $\phi$ in the $L^2(\mathcal{T})$ norm, which will constitute a first criterion concerning the numerical stability of our scheme. 
 
 Since a ground state $\varphi$ is defined as the solution of the minimization
   problem above with the mass constraint, we may write the corresponding Euler-Lagrange equation:
 There exists $\mu\in\R$ such that
\begin{equation*} \label{nonlinear_eigenvalue_problem}
  \left\{
      \begin{aligned}
 \mu \varphi(x) =  - \frac{1}{2m}\Delta \varphi(x) + V(x) \varphi(x) + \gamma |\varphi(x)|^2 \varphi(x), & \quad x \in \Omega,\\
       \varphi(x)=0, & \quad x \in \partial \Omega, \\
        \int_{\Omega} |\varphi(x)|^2 \dd x = 1 . &
      \end{aligned}
    \right.
  \end{equation*}

We adopt the same minimization strategy in the discretized context, defining the discrete counterpart of the energy $E$ for a vector $U^n \in \C^N$, namely
\[ E_{\mathcal{T}}(U)= -\frac{1}{2m} \langle U, A_{\mathcal{T}} U \rangle_{\mathcal{T}} + \langle U, V_{\mathcal{T}} U \rangle_{\mathcal{T}} + \frac{\gamma}{2} \langle U, |U|^2 \cdot U \rangle_{\mathcal{T}},   \]
which is associated with the discrete gradient 
\[ \nabla_U E_{\mathcal{T}} (U) = -\frac{1}{m} A_{\mathcal{T}} U +2 V_{\mathcal{T}} \cdot U + 2 \gamma |U|^2 \cdot U.   \]
Note here that the vector $V_{\mathcal{T}}\in \C^N$ designs the vector of values taken by the potential $V$ at each circumcenter of the triangulation $\mathcal{T}$. Also, the vector product~"$\cdot$" has here to be understood pointwise, as well the modulus $|U|^2=U\cdot \overline{U}$. We then perform a semi-implicit imaginary time discretization using an explicit Euler method for the laplacian part, based on a recent convergence result proved in \cite{faou2018}, and that can be written under the form
\begin{equation} \label{normalized_gradient_flow_splitting}
  \left\{
      \begin{aligned}
        & \frac{U^{n+\frac{1}{2}}-U^n}{\kappa} = \frac{1}{m} A_{\mathcal{T}} U^{n+\frac12} - 2V_{\mathcal{T}} \cdot U^{n+\frac{1}{2}} - 2\gamma |U^n|^2 U^{n+\frac{1}{2}}, \\
       & U^{n+1}= \frac{U^{n+\frac{1}{2}}}{\| U^{n+\frac{1}{2}}\|_{L^2(\mathcal{T})}},
      \end{aligned}
    \right.
\end{equation}
with a small stepsize $\kappa >0$ and starting from some initial data $U^0$. At each iteration, we check if
\[ E_{\mathcal{T}} \left(U^{n+1} \right) < E_{\mathcal{T}} \left( U^{n} \right).  \]
If it is, we move on, but if it is not, we come back to $U^{n}$ and we put $\kappa\leftarrow \kappa/2$.  The algorithm stops at step $n \in \N$ as soon as the condition 
\begin{equation} \label{stopping_criterion}
\mathcal{C}_{\mathcal{T}}(U^n) := \left\| \nabla_U E_{\mathcal{T}} (U^n) - \langle  \nabla_U E_{\mathcal{T}} (U^n), U^n \rangle_{L^2(\mathcal{T})}  U^n  \right\|_{L^2(\mathcal{T})} \leq \epsilon
\end{equation}  
is fulfilled for some vector $U^n$, and for a fixed given threshold $\epsilon >0$.

For a rigorous proof about the fact that equation \eqref{stopping_criterion} defines a stopping criterion (typically that $\mathcal{C}_{\mathcal{T}}(U) \rightarrow 0$ as $U\rightarrow U^*$, where $U^*$ designs a local minimum of the energy $E_{\mathcal{T}}$), we refer to the proof of \cite[Proposition 3]{dujardin2022}, which essentially relies on the fact that if the discrete energy~$E_{\mathcal{T}}$ admits a local minimum $U^*$ on the sphere, the function 
\[ f : t \mapsto E_{\mathcal{T}} \left( \frac{U^*+t U}{\| U^*+t U \|_{L^2(\mathcal{T})}}  \right)  \]
is smooth and has a null derivative at $t=0$.

\section{Post-processing algorithms} \label{section_post_processing}

\subsection{Vortex detection} \label{subsection_vortex_detection_post}
We describe the algorithm for the detection of vortices and the computations of their indices, adapting the cartesian-grid method of \cite[Section 4.1]{dujardin2022} on a triangulation. The algorithm relies on 3 numerical parameters $tol_1>0$, $tol_2>0$ and $\lambda_{\max}\in \N^*$. It follows the four steps below, where the three first steps identify the vortices' centers and the last step computes the vortices indices. In this section, $\psi$ denotes the square complex matrix with $N^2$ entries (where we recall that $N$ denotes the number of triangles on our triangulation) of the Bose-Einstein condensate.
\begin{itemize}
	\item \underline{Step 1:} We determine the potential centers of the vortices and establish a list of candidates called \textbf{potential\_vortices} such that $|\psi(n)|^2<tol_1$.
	\item \underline{Step 2:} We build a second list $\mathbb{P}$ based on the first list above using the following rule. For
each potential vortex number $n$ in the list establish in Step 1, we consider the values of $|\psi|^2$ on the adjacent triangles (let's say triangles that are at distance $\lambda=1$, denoted $\mathbb{S}_{\lambda}(n)$). If the values of $|\psi|^2$ at all points	of these adjacent triangles are such that $|\psi|^2-|\psi(n)|^2 > tol_2$, then we add the vortex center $n$ to the second list $\mathbb{P}$, and we set $\lambda=1$ as the characteristic length of the potential vortex. If not, we repeat this procedure for triangles at distance $\lambda=2,\ldots,\lambda_{\max}$. To summarize, we have determined a list $\mathbb{P}$ of points $n$ satisfying the
conditions
\[ |\psi(n)|^2 < tol_1 \quad \text{and} \quad |\psi(j)|^2 > |\psi(n)|^2 + tol_2   \]
for all triangles $j\in \mathbb{S}_{\lambda}(n)$ at distance $1 \leq \lambda \leq \lambda_{\max}$.
	\item \underline{Step 3:} We consider each vortex center from the list $\mathbb{P}$ and we identify if another center is
inside the set $ \bigcup_{\lambda=1}^{\lambda_{\max}} \mathbb{S}_{\lambda}(n)$. If this is the case, we eliminate the center with the biggest $|\psi|^2$ from the list $\mathbb{P}$. We repeat this step until we are left with isolated centers.
	\item \underline{Step 4:} We compute the indices using the following rules. For each $n \in \mathbb{P}$, we compute the angles $\omega(j) \in \left[-\pi,\pi \right]$ formed between the circumcenter of the triangle $n$ and the circumcenters of the triangles $j \in \mathbb{S}_{\lambda}(n)$ and the abscissa line. We then resort these angles in increasing order. We first compute the angle $\theta_0=\arg(\psi(j))$ associated to the triangle $j$ of lowest angle. After computing the first angle $\theta_0$, we proceed to compute the other angles $\theta_1,\theta_2,\ldots,\theta_M$, with $M(n)=\Card(\mathbb{S}_{\lambda}(n))$, the following way:
	\begin{itemize}
		\item After computing the angle $\theta_m$, the next angle $\tilde{\theta}_{m+1}$ is computed as an argument of the next value of $\psi$ on the set $\mathbb{S}_{\lambda}(n)$ with anticlockwise rotation.
		\item We set $\theta_{m+1}:=\tilde{\theta}_{m+1}+2k\pi$ with $k=\argmin_{l \in \Z} |\tilde{\theta}_{m+1}-\theta_m+2\pi l  |$.
	\end{itemize}
	Eventually, the index of the vortex $n$ is equal to $(\theta_{M(n)}-\theta_0)/2\pi$. This last step is illustrated at Figure \ref{tikz_triangulation}.
\end{itemize}

\begin{center}
\begin{figure} 
\begin{tikzpicture}[line cap=round,line join=round,x=5cm,y=5cm] 
\clip(-0.8124846590133511,-0.36477869720517864) rectangle (1.2887017626879362,0.9119511688958221);

\draw [line width=1pt] (-0.2537759948061049,0.2067675022527642)-- (0,0.5);
\draw [line width=1pt] (0,0.5)-- (0.09823660325541984,0.12463122937174202);
\draw [line width=1pt] (0.09823660325541984,0.12463122937174202)-- (-0.2537759948061049,0.2067675022527642);

\draw [line width=0.5pt] (0.6145217470789895,0.552913223679929)-- (0,0.5);
\draw [line width=0.5pt] (0.09823660325541984,0.12463122937174202)-- (0.6145217470789895,0.552913223679929);
\draw [line width=0.5pt] (-0.3359122676871274,0.6643838797327447)-- (-0.2537759948061049,0.2067675022527642);
\draw [line width=0.5pt] (0,0.5)-- (-0.3359122676871274,0.6643838797327447);
\draw [line width=0.5pt] (-0.2537759948061049,0.2067675022527642)-- (-0.35937977422456235,-0.35645265464567344);
\draw [line width=0.5pt] (-0.35937977422456235,-0.35645265464567344)-- (0.09823660325541984,0.12463122937174202);

\draw [line width=0.4pt,dash pattern=on 1pt off 1pt] (0,0.5)-- (0,0.8990589451070937);
\draw [line width=0.4pt,dash pattern=on 1pt off 1pt] (0,0.8990589451070937)-- (0.6145217470789895,0.552913223679929);
\draw [line width=0.4pt,dash pattern=on 1pt off 1pt] (-0.3359122676871274,0.6643838797327447)-- (0,0.8990589451070937);
\draw [line width=0.4pt,dash pattern=on 1pt off 1pt] (0.09823660325541984,0.12463122937174202)-- (0.44438232468258587,-0.24498199859285766);
\draw [line width=0.4pt,dash pattern=on 1pt off 1pt] (0.44438232468258587,-0.24498199859285766)-- (0.6145217470789895,0.552913223679929);
\draw [line width=0.4pt,dash pattern=on 1pt off 1pt] (-0.35937977422456235,-0.35645265464567344)-- (0.44438232468258587,-0.24498199859285766);
\draw [line width=0.4pt,dash pattern=on 1pt off 1pt] (-0.3359122676871274,0.6643838797327447)-- (-0.7348598788235221,0.2302350087901991);
\draw [line width=0.4pt,dash pattern=on 1pt off 1pt] (-0.7348598788235221,0.2302350087901991)-- (-0.2537759948061049,0.2067675022527642);
\draw [line width=0.4pt,dash pattern=on 1pt off 1pt] (-0.7348598788235221,0.2302350087901991)-- (-0.35937977422456235,-0.35645265464567344);

\begin{scriptsize}
\draw [color=uuuuuu] (-0.049586650369565756,0.28648384612272554)-- ++(-1.5pt,-1.5pt) -- ++(3pt,3pt) ++(-3pt,0) -- ++(3pt,-3pt);
\draw [color=uuuuuu] (0.3196042398849475,0.38310365535924157)-- ++(-1.5pt,-1.5pt) -- ++(3pt,3pt) ++(-3pt,0) -- ++(3pt,-3pt);
\draw [color=uuuuuu] (0.5054366127355339,0.15908654836127364)-- ++(-1.5pt,-1.5pt) -- ++(3pt,3pt) ++(-3pt,0) -- ++(3pt,-3pt);
\draw [color=uuuuuu] (0.03927786208228369,-0.277474821298121)-- ++(-1.5pt,-1.5pt) -- ++(3pt,3pt) ++(-3pt,0) -- ++(3pt,-3pt);
\draw [color=uuuuuu] (-0.50679612760209,-0.03730165561768766)-- ++(-1.5pt,-1.5pt) -- ++(3pt,3pt) ++(-3pt,0) -- ++(3pt,-3pt);
\draw [color=uuuuuu] (-0.48539732313102524,0.40137383603914245)-- ++(-1.5pt,-1.5pt) -- ++(3pt,3pt) ++(-3pt,0) -- ++(3pt,-3pt);
\draw [color=uuuuuu] (-0.11053519177975404,0.6995294725535468)-- ++(-1.5pt,-1.5pt) -- ++(3pt,3pt) ++(-3pt,0) -- ++(3pt,-3pt);
\draw [color=uuuuuu] (0.2923584831271245,0.6995294725535469)-- ++(-1.5pt,-1.5pt) -- ++(3pt,3pt) ++(-3pt,0) -- ++(3pt,-3pt);
\draw [color=uuuuuu] (-0.23439558146699643,0.4464254306967888)-- ++(-1.5pt,-1.5pt) -- ++(3pt,3pt) ++(-3pt,0) -- ++(3pt,-3pt);
\draw [color=uuuuuu] (-0.14113430549293735,-0.10586324726315403)-- ++(-1.5pt,-1.5pt) -- ++(3pt,3pt) ++(-3pt,0) -- ++(3pt,-3pt);
\end{scriptsize}

\draw [<->] (1,0.8) -- (1,0.7) -- (1.1,0.7);
\draw (1.1,0.7) node[below] {$x_1$} ;
\draw (1,0.8) node[below left] {$x_2$} ;

\draw[->,>=latex] (0.75,0) to[bend right] (0.75,0.5);

\draw (-0.049586650369565756,0.28648384612272554) node[below right] {$n$} ;
\draw (-0.50679612760209,-0.03730165561768766) node[below right] {$\theta_0$} ;
\draw (-0.14113430549293735,-0.10586324726315403) node[above left] {$\theta_1$} ;
\draw (0.03927786208228369,-0.277474821298121) node[above] {$\theta_2$} ;
\draw (0.5054366127355339,0.15908654836127364) node[below left] {$\theta_3$} ;
\draw (0.3196042398849475,0.38310365535924157) node[left] {$\theta_4$} ;
\draw (0.2923584831271245,0.6995294725535469) node[left] {$\theta_5$} ;
\draw (-0.11053519177975404,0.6995294725535468) node[above] {$\theta_6$} ;
\draw (-0.23439558146699643,0.4464254306967888) node[above] {$\theta_7$} ;
\draw (-0.48539732313102524,0.40137383603914245) node[below] {$\theta_8$} ;

\end{tikzpicture}
\caption{The set $\mathbb{S}_{1}(n)$ with angles $\theta_0$ to $\theta_8$.}
\label{tikz_triangulation}
\end{figure}
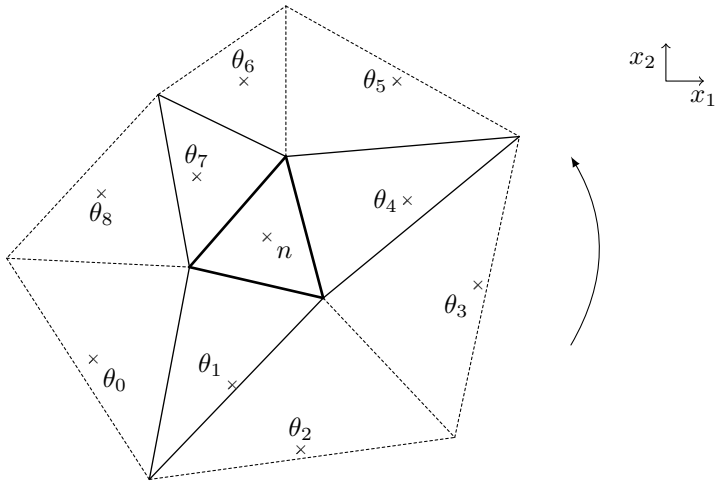
\end{center}

\subsection{Decomposition of the dynamics in the eigenbasis of the linearized operator} \label{subsection_decompo}
As we inject energy in our quantum system through the rotational potential modulation $V_{\rot}$, we are interested in how this energy is distributed along radial and azimutal modes, in particular during vortex nucleation. Note that in a turbulent regime, one should expect a cascade of energy (towards low or high frequencies according to the dimension), inducing a scaling distribution of the occupation of the different momentum modes.

In order to compute these modes, we consider the stationary linear part of equation \eqref{GP}, namely the linear operator
\[ H=-\frac{1}{2m} \Delta + V_{\pot}.  \]
From classical Sturm-Liouville theory, there exists an increasing sequence of eigenvalues associated to normalized eigenfunctions of $H$ which form an orthonormal basis of $L^2(\Omega)$. From the radial symmetry of the Gaussian potential $V$ and Fourier decomposition, we infer that these eigenfunctions can be written in polar coordinates under the form
\begin{equation} \label{polar_decomposition_eigenfunctions}
\phi_{p,l}(r,\theta)=\varphi_p(r) e^{i \ell \theta}, \quad p \in \N, \ \ell \in \Z, 
\end{equation}  
where the $(\varphi_p)_{p \in \N}$ are solutions of the one-dimensional Sturm-Liouville problem with Dirichlet boundary conditions
\begin{equation} \label{one_dimensional_gaussian_potential_eq}
\left\{ 
\begin{aligned}
& -\frac{1}{2m} \left( \partial^2_r \varphi_p(r) + \frac{1}{r} \partial_r \varphi_p(r) -\frac{\ell^2}{r^2} \varphi_p(r) \right)  - V_0 e^{-2m(r-1)^2} \varphi_p(r)= \lambda_{p,l}\varphi_p(r)  \\
& \varphi_p(r_{\min})=\varphi_p(r_{\max})=0,
\end{aligned}
\right.
\end{equation}
for all $r \in (r_{\min},r_{\max})$. Fixing two integers $P,L \in \N$, in order to numerically approach the first eigenfunctions $(\phi_{p,l})_{0 \leq p \leq P, -L \leq \ell \leq L}$, we proceed as follow:
\begin{itemize}
	\item We discretize by standard finite differences with Dirichlet conditions the left hand side of equation \eqref{one_dimensional_gaussian_potential_eq}, with fixed $\ell \in \left\{ -L, \ldots , L \right\}$. More precisely, we fix a number of points $\mathbf{n} \geq 2$, and we denote by $\mathbf{r}=( \mathbf{r}_1, \ldots, \mathbf{r}_{\mathbf{n}})^{\top}$ the discrete vector defined by
\[ \mathbf{r}_k = r_{\min}+(r_{\max}-r_{\min})\frac{k}{\mathbf{n}}, \quad k=0,\ldots,\mathbf{n},  \]
where the $(\mathbf{r}_k)_k$ are equidistants points at a range $\mathbf{h}=\frac{r_{\max}-r_{\min}}{\mathbf{n}}$. Denoting $\mathcal{R}=\diag(\mathbf{r}_0,\ldots,\mathbf{r}_\mathbf{n})$, we then define the matrix $\mathcal{H} \in \mathcal{M}_{\mathbf{n}+1}(\C)$ by
\[ \mathcal{H} = -\frac{1}{2m} \left( \mathcal{L} + \mathcal{R}^{-1} \times \mathcal{D} - \ell^2 \mathcal{R}^{-2} \right) -V_0 \diag \left( e^{-2m(\mathbf{r}_0-1)^2},\ldots, e^{-2m(\mathbf{r}_{\mathbf{n}}-1)^2} \right), \]
where the discrete laplacian matrix $\mathcal{L}$ and the first order centered finite difference matrix $\mathcal{D}$, with Dirichlet conditions, are given by the formulas
\[ \mathcal{L} = -\frac{1}{\mathbf{h}^2} \left( \begin{array}{ccccc}
													 0 & \ldots & 0 & &  (0) \\
													-1 & 2 & -1 & &   \\
													 & \ddots & \ddots &\ddots &    \\
														& & -1 & 2 & -1 \\
													(0) & & 0 & \ldots & 0 
												\end{array} \right)
	\ \text{and} \ 											
\mathcal{D} = \frac{1}{2\mathbf{h}} \left(  \begin{array}{ccccc}
													 0 & \ldots & 0 & &  (0) \\
													-1 & 0 & 1 & &   \\
													 & \ddots & \ddots &\ddots &    \\
														& & -1 & 0 & 1 \\
													(0) & & 0 & \ldots & 0 
												\end{array} \right)	.		 \]

	\item We use the function \textbf{eigs} from \textsc{GNU Octave} in order the get the first $P+1$ eigenvectors of the matrix $\mathcal{H}$ (note here that we obviously need that $P+1 \leq \mathbf{n}$, which will not be an issue as we will take $\mathbf{n}=500$ in our computation, and only seek for the few firsts eigenvectors). At this stage, we have $P+1$ vectors $\boldsymbol\varphi_0,\ldots,\boldsymbol\varphi_P \in \C^{\mathbf{n}+1}$.
	\item We then use the function \textbf{interp1} from \textsc{GNU Octave} to interpolate on the circumcenters of the triangles of the triangulation $\mathcal{T}$ using equation \eqref{polar_decomposition_eigenfunctions}, so we get $(P+1)\times (2L+1)$ vectors $(\Phi_{p,\ell})_{0 \leq p \leq P, -L \leq \ell \leq L}$ of size $N$, which we normalize on $L^2(\mathcal{T})$, namely for all $p,\ell,$
	\[ \| \Phi_{p,\ell} \|_{L^2(\mathcal{T})}^2 = \sum_{K \in \mathcal{T}} |\Phi_{p,\ell}(K)|^2 |K|=1.  \]
\end{itemize}

Note that the vectors $(\Phi_{p,\ell})_{p,\ell}$ are quasi-orthogonal for the $L^2(\mathcal T)$-scalar product, in the sense that $\langle \Phi_{p,\ell}, \Phi_{p',\ell'} \rangle \simeq \delta_{p,p'} \delta_{\ell,\ell'}$. For instance, from the numerical point of view, taking the same set of parameters as in the forthcoming Section \ref{simulations_section}, these quantities are around $10^{-4}$ in absolute value for  $\Phi_{p,\ell} \neq \Phi_{p',\ell'}$. Hence for a vector $U \in \C^{N}$, we call its decomposition on the set $(\Phi_{p,\ell})_{p,\ell}$ the quantity
\[ \sum_{p =0}^P \sum_{\ell =-L}^L \langle U, \Phi_{p,\ell} \rangle_{\mathcal{T}}  \Phi_{p,\ell}. \]
We also emphasize that in the following numerical simulations, we will have $N \gg(P+1)\times(2L+1)$, as the number of triangles $N$ will be somewhat between $10^4$ and $10^5$, whereas $(P+1)\times(2L+1)$ remains between $10^2$ and $10^3$.

\section{Vorticity and pseudo-vorticity}
In order to numerically track vortices nucleation, one may also rely on particular physical quantities related to the evolution of the wave function $\psi$. A strong candidate would naturally be the vorticity $\omega = \nabla \times v$ of the quantum fluid, taking the scalar cross product of the velocity
\[v=\frac{1}{2i} \frac{ (\overline{\psi}\nabla \psi -\psi \nabla \overline{\psi})}{|\psi|^2}, \]
of the wave function $\psi$, so that $\omega : \R_+ \times \R^2 \mapsto \R$. The vorticity of the quantum fluid is known to be represented by a distribution of Dirac's $\delta$-functions at the core of each vortices, making this quantity very appealing for the numerical detection of vortices throughout the dynamics. However, such distributions leads to high gradients which may be very  unstable to approximate, so one has to adapt this strategy in order to perform reliable numerical simulations. In the following we introduce two vorticity-based quantities, namely the \textit{regularized vorticity} and the \textit{pseudo-vorticity}, which are used in the quantum turbulence literature (see for instance \cite{brachet2023} or \cite{krstulovic2016}), and we briefly compare them with the post-processing algorithm of Section \ref{subsection_vortex_detection_post}.

\subsection{Regularized vorticity}  \label{subsection_vortex_reg_vorticity}
 In order to regularize the vorticity , one can fix a saturation parameter $\delta >0$, so that we define the regularized velocity
\[ v_{\delta}= \frac{1}{2i} \frac{( \overline{\psi}\nabla \psi -\psi \nabla \overline{\psi})}{|\psi|^2+\delta}  \]
which leads to the regularized vorticity
\[ \omega_{\delta} = \nabla \times v_{\delta}.   \]
 In order to numerically compute the regularized vorticity, we first compute $\nabla \psi$ as for $\nabla \phi$ which provides the discrete velocity, and we then rely on the work of Delcourte, Domelevo and Omnes \cite{delcourte2007} to compute the discrete scalar cross product on an admissible triangulation based on DDFV schemes, which can be describes as follows: for a triangle $K \in \mathcal{T}$ we approximate 
\[ \omega_{\delta}(K)=\frac{1}{|K|} \int_{K} (\nabla \times v_{\delta}) \simeq \frac{1}{|K|} \sum_{\sigma \in \mathcal{E}_{K}} |\sigma| v_{\delta}(K) \cdot \tau_{\sigma}, \]
where $\tau_{\sigma}$ is the normalized vector such that $\tau_{\sigma}$ and $n_{\sigma}$ are orthonormal positively oriented bases of $\R^2$. We describe the discretization of the gradient appearing in the expression of $v_{\delta}$ in the next section.

\subsection{Pseudo-vorticity} \label{subsection_vortex_pseudo_vorticity}
We also consider another quantity called the pseudo-vorticity, which was recently analyzed by the first author in \cite{chauleur2024FVGP}: denoting the density current $J=\frac{1}{2i} ( \overline{\psi}\nabla \psi -\psi \nabla \overline{\psi} )$, we define 
\[ \omega_{\ps}=\frac12 \nabla \times J = \nabla (\Re \psi) \times \nabla (\Im \psi).   \]
This last formulation will be use in order to numerically compute this quantity, relying on a piecewise constant approximation of the gradient of $\psi$ introduced in \cite{eymard2006}, and which relies on the use of \textit{diamond points} of the triangulation, namely the point $x_{\sigma}$ defined at the middle of each edges $\sigma \in \mathcal{E}_{K}$. Then, on each triangle $K \in \mathcal{T}$, we approximate
\[ \frac{1}{|K|} \int_{K} \nabla \psi \simeq \frac{1}{|K|} \sum_{\sigma \in \mathcal{E}_{K}} d\psi_{K,\sigma} (x_{\sigma}-x_{K}),   \]
where
\begin{equation*} 
d\psi_{K,\sigma} =   \left\{
      \begin{aligned}
        & \tau_{\sigma} (\psi_L-\psi_{K}), &\quad \ \text{if} \ \sigma=K|L \in \mathcal{E}_{\interior},\\
        & -\tau_{\sigma} \psi_{K}, &\quad \ \text{if} \ \sigma \in \mathcal{E}_{\exterior},
      \end{aligned}
    \right.
\end{equation*}
and
\begin{equation*} 
  \left\{
      \begin{aligned}
        & \tau_{\sigma} =\frac{|\sigma|}{|x_{K}-x_L|}, &\quad \ \text{if} \ \sigma=\kappa|L \in \mathcal{E}_{\interior},\\
        & \tau_{\sigma} =\frac{|\sigma|}{d(x_{K},\partial \Omega)}, &\quad \ \text{if} \ \sigma \in \mathcal{E}_{\exterior}.
      \end{aligned}
    \right.
\end{equation*}

This approximation of the gradient of $\psi$ on the triangulation $\mathcal{T}$ will also be use in order to compute the velocity $v$ appearing in the definition of the vorticity.

\subsection{Detection of vortices}
Both the regularized vorticity and the pseudo-vorticity appear to be smooth functions which have a local extremum at each vortex core and are small outside vortices. One has then to put a threshold in order to detect these local extremum and to perform a root-finding routine (typically a Newton method) in the neighborhood of the vortex in order to recover their exact position.
Note that the sign of the local extremum determines the sign of the charge of the vortex. We also emphasize that the pseudo-vorticity only requires the computation of one discrete gradient, instead of two with for the regularized vorticity, which makes it a bit more appealing from the computational efficiency point of view. A more precise comparison of the computation of these two quantities with the post-processing algorithm described in Section \ref{subsection_vortex_detection_post} will be made in the following section. On the other hand, this algorithm return the precise indice of each vortices, and not only their sign, even if it is now a common known feature that stable vortices can only have $\pm 1$ indices.

\section{Numerical simulations} \label{simulations_section}

\subsection{Accuracy tests} This first section is devoted to numerically show the different orders of convergence theoretically announced in the previous sections, as well as the discrete ground state stability. We perform all incoming simulations (if not mentioned otherwise) with the set of parameters $m=10$, $V_0=100$, $\gamma=100$, $r_{\min}=0.6$ and $r_{\max}=1.4$, on a triangulation $\mathcal{T}$ composed of $N=39852$ triangles generated via the arguments given in Section \ref{triangulation_section} with a stepsize control parameter $h =0.03$, which leads respectively to a number of circles and a number of point per circle $N_c=41$ and $N_p=486$ via Proposition \ref{proposition_admissible_triangulation}. 

\subsubsection{Space order accuracy} 
Eigenfuntions of the Laplace operator $-\Delta$ on the ring $\Omega$ with Dirichlet conditions are explicitly known in terms of Bessel functions of the first kind $J_{\alpha}$ and of the second kind $Y_{\alpha}$ \cite[Section 3.2]{grebenkov2013}. In polar coordinates, these eigenfunctions 
\begin{equation} \label{eigenfunctions_laplacian_ring}
u_{\alpha,\beta,\sigma}(r,\theta)=\left[ J_{\alpha} \left(r \sqrt{\lambda_{\alpha,\beta}}  \right) +c_{\alpha,\beta} Y_{\alpha} \left( r \sqrt{\lambda_{\alpha,\beta}} \right) \right] \times \left\{ \begin{aligned}
      &  \cos(\alpha \theta), \quad \sigma=1, \\
       & \sin(\alpha \theta), \quad \sigma=2 \ (\alpha\neq 0),
      \end{aligned}
    \right.
\end{equation}
are enumerated with respect to the triplet 
\[(\alpha,\beta,\sigma) \in \N \times \N^* \times \left\{1,2 \right\},\]
and associated to the eigenvalues $(\lambda_{\alpha,\beta})_{\alpha,\beta}$ which are simple for $\alpha=0$ and twice degenerate for $\alpha>0$. The coefficients $(\lambda_{\alpha,\beta})_{\alpha,\beta}$ and $(c_{\alpha,\beta})_{\alpha,\beta}$ are determined by the Dirichlet boundary conditions at $r=r_{\min}$ and $r=r_{\max}$,
\[  \left\{  \begin{array}{cc}
				J_n \left(r_{\min} \sqrt{\lambda_{\alpha,\beta}}  \right), & Y_n \left(r_{\min} \sqrt{\lambda_{\alpha,\beta}}   \right) =0,\\
				J_n \left(r_{\max} \sqrt{\lambda_{\alpha,\beta}}  \right),  & Y_n \left(r_{\max} \sqrt{\lambda_{\alpha,\beta}}   \right)=0,
					\end{array} \right.  \]
so for all $\alpha \geq 0$ we look at the zeros of the function
\[ f_{\alpha}(x)=  \left|  \begin{array}{cc}
				J_{\alpha} \left(r_{\min} x \right) & Y_{\alpha} \left( r_{\min} x \right) \\
				J_{\alpha} \left(r_{\max} x \right)  & Y_{\alpha} \left(r_{\max} x \right)
					\end{array} \right|, \]
which gives the square root of the eigenvalues $(\sqrt{\lambda_{n,k}})_k$. We also easily get the coefficients $(c_{\alpha,\beta})_{\beta}$ by the direct computation
\[  c_{\alpha,\beta}=-\frac{J_{\alpha}(r_{\max}\sqrt{\lambda_{\alpha,\beta}})}{Y_{\alpha}(r_{\max}\sqrt{\lambda_{\alpha,\beta}})}.  \]
In order to corroborate the first order convergence in space of the Finite Volume method described in Section \ref{finite_volume_section}, we now focus on eigenfunctions for $\alpha=0$. Bessel functions of the first kind $J_{\alpha}$ and of the second kind $Y_{\alpha}$ are respectively implemented in \textsc{GNU Octave} via the functions \textbf{besselj} and \textbf{bessely}, and the first zeros of the function $f_{\alpha}$ are numerically approached using the function \textbf{fzero}.  Using the expression \eqref{eigenfunctions_laplacian_ring} for $\alpha=0$, $\beta \in \N^*$ and $\sigma=1$, this gives a vector $U_{0,\beta,1} \in \C^N$ on the triangulation $\mathcal{T}$ associated to the eigenvalues $\lambda_{0,\beta}$. As $A_{\mathcal{T}} \rightarrow -\Delta$ as $h\rightarrow 0$, we plot in Figure \ref{fig:space_order_convergence_figure.png} the evolution of the norms
\[ \|-A_{\mathcal{T}} U_{0,\beta,1} -\lambda_{0,\beta} U_{0,\beta,1} \|_{L^2(\mathcal{T})}  \]
as $h\rightarrow 0$, for $\beta=1,2,3$.

\begin{figure}[h]
	\centering
		\includegraphics[width=0.6\textwidth,trim = 0cm 0cm 0cm 0cm, clip]{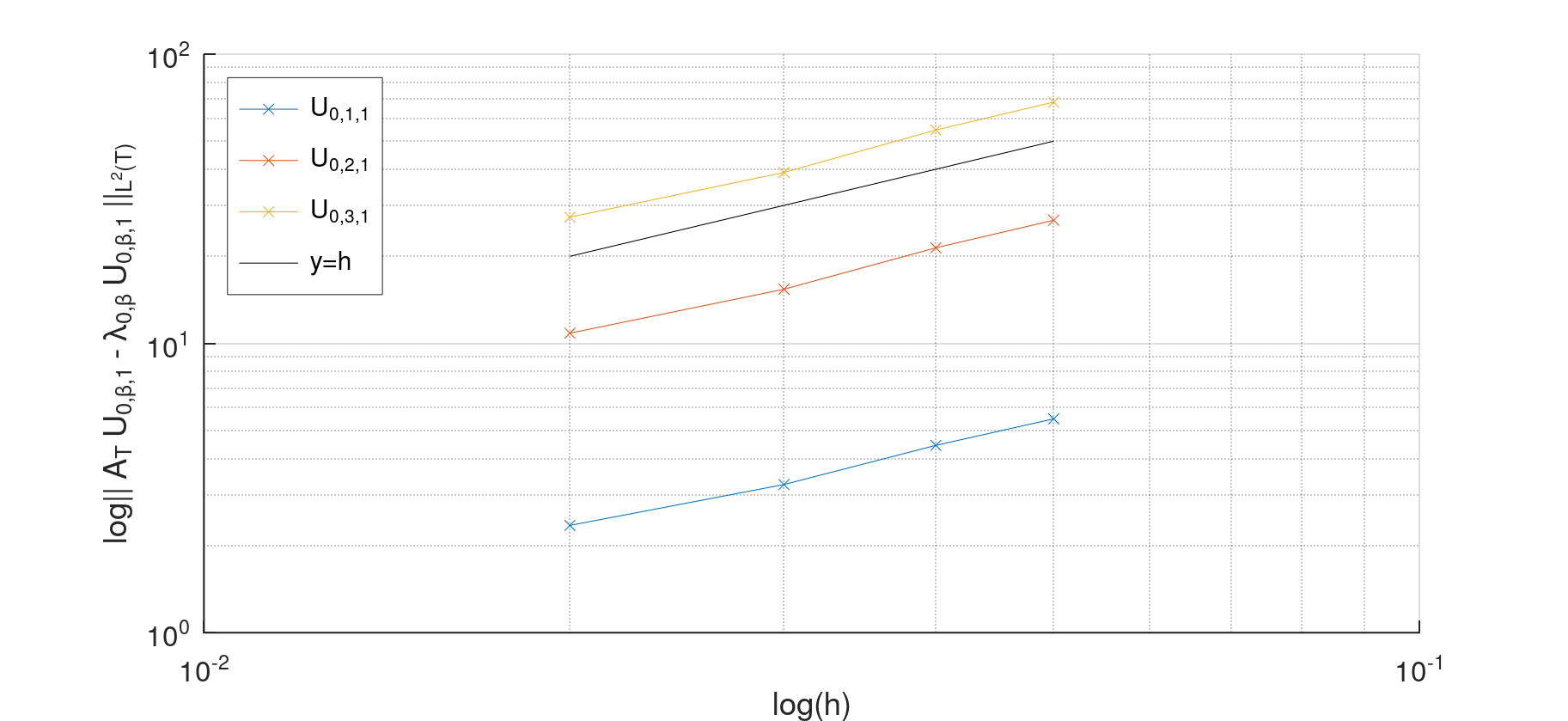}	
	\caption{Convergence of $A_{\mathcal{T}}$ on the triangulation $\mathcal{T}$ towards $\Delta$.}
	\label{fig:space_order_convergence_figure.png}
\end{figure}

\subsubsection{Time order accuracy}
To corroborate the second order convergence of the Strang splitting method, we now compute the dynamics of equation \eqref{GP} in the case without rotation $V_{\rot}=0$. We refer to equation \eqref{definition_strang_splitting} for the definition of the Strang splitting $\Phi^{\tau,t}$ in our setting, and we denote by $\mathbf{m}_{\Phi} \in \N$ the order of this splitting method, so that there exists a constant $\mathbf{e}>0$ such that
\[ \Phi^{T_{\max}}_{\tau} (\psi_0) = \psi(T_{\max},\cdot) + \mathbf{e} \tau^{\mathbf{m}_{\phi}} + \mathcal{O} \left( \tau^{\mathbf{m}_{\phi}+1} \right);  \]
where $\psi(T_{\max},\cdot)$ denotes the solution of equation \eqref{GP} at time $t=T_{\max}$ with initial condition $\psi(0,\cdot)=\psi_0$, and where we use the convention
\[ \Phi^{T_{\max}}_{\tau}= \Phi^{\tau,(j-1)\tau} \circ \ldots \circ  \Phi^{\tau,0}\]
as $T_{\max}=J \tau$. Denoting 
\begin{equation*}
\left\{
\begin{aligned}
& y_{\tau}=\Phi_{\tau}^{T_{\max}}(\psi_0)= \psi(T_{\max}) +\mathbf{e} \tau^{\mathbf{m}_{\phi}} + \mathcal{O}\left(\tau^{\mathbf{m}_{\phi}+1} \right),\\
& y_{2\tau}=\Phi_{2\tau}^{T_{\max}}(\psi_0)= \psi(T_{\max}) +\mathbf{e}(2 \tau)^{\mathbf{m}_{\phi}} + \mathcal{O}\left(\tau^{\mathbf{m}+1} \right),\\
& y_{\frac{\tau}{2}}=\Phi_{\frac{\tau}{2}}^{T_{\max}}(\psi_0)= \psi(T_{\max}) +\mathbf{e}\left( \frac{\tau}{2}\right)^{\mathbf{m}_{\phi}} + \mathcal{O}\left(\tau^{\mathbf{m}+1} \right),
\end{aligned}
\right.
\end{equation*}
a classical estimate for $\mathbf{m}_{\phi}$ is
\[ \mathbf{m}_{\phi} = \log \left( \frac{\| y_{2\tau} - y_{\tau} \|_{L^2(\mathcal{T})}}{ \| y_{\tau} - y_{\frac{\tau}{2}} \|_{L^2(\mathcal{T})}}   \right)\log(2)^{-1} + \mathcal{O}\left(\tau \right)  \]
which follows from standard computations (see \cite{lubich2006}). In order to show that the Strang splitting defined by \eqref{definition_strang_splitting} is well of order two (so that $\mathbf{m}_{\phi}=2$), we perform several simulations of \eqref{GP} with $T_{\max}=0.1$ and $J=2^\mathbf{k}$ for varying $\mathbf{k}\in \N^*$ with respectively $\tau$, $2 \tau$ and $\frac{\tau}{2}$ for the computations of $y_{\tau}$, $y_{2\tau}$ and $y_{\frac{\tau}{2}}$, which we reports in the following table.

\begin{center}
\begin{tabular}{|c|cccccc|} 
  \hline
$\mathbf{k}$ & 5 & 6 & 7 & 8 & 9 & 10 \\
\hline
 $\mathbf{m}_{\phi}$ &  1.8237 &  1.9321 & 1.9861 & 2.0022 & 2.0038 & 2.0021 \\
 \hline   
\end{tabular}
\end{center}

\subsubsection{Numerical stability of the ground state}
 We perform the normalized gradient flow algorithm detailed in Section \ref{GS_section}, starting the descent gradient method from a unitary normalized vector with an imaginary time step $\kappa=1.10^{-2}$ and a threshold $\epsilon=5.10^{-3}$. We then compute the dynamics of equation \eqref{GP} via the Strang Splitting detailed in Section \ref{splitting_section}, starting from the approximation of the ground state $U_{\GS}$, over a maximum time $T_{\max}=3$, so we take $J=5000$ discretization points in time, with $V_{\rot}=0$. Note that the linear flow $\Phi_A^{\tau}$ is computed using a second order Padé approximant
 \[ e^{i\frac{\tau}{2m}A_{\mathcal{T}}} \simeq \left( I_N-i\frac{\tau}{4m}A_{\mathcal{T}} \right)^{-1}\left(I_N+i\frac{\tau}{4m}A_{\mathcal{T}} \right) ,   \]
 where $I_N \in \mathcal{M}_{N}(\C)$ denotes the identity matrix. Moreover, the matrix $A_{\mathcal{T}}$ is a sparse matrix, so that the resolution of the linear system
 \[ \left( I_N-i\frac{\tau}{4m}A_{\mathcal{T}} \right)X=B  \]
 can be efficiently precomputed using its sparse $LU$ decomposition with the \textsc{GNU Octave} function \textbf{lu}. Denoting $\Phi_{\tau}^{t}(\mathbf{\psi}_0)$ the vector solution at time $t$, we then plot the evolution of the error
 \[ \err_{\GS}(t)=\left\| |\Phi_{\tau}^{t}(U_{\GS})| -  U_{\GS} \right\|_{L^2(\mathcal{T})} \]
 in Figure \ref{fig:evolution_ground_state.png}. We observe that this quantity is bounded by $2. 10^{-4}$, so the approximation of the ground state $\psi_0$ is particularly stable under the dynamic of $\Phi^{\tau}$.
 
\begin{figure}[h]
	\centering
		\includegraphics[width=0.6\textwidth,trim = 0cm 0cm 0cm 0cm, clip]{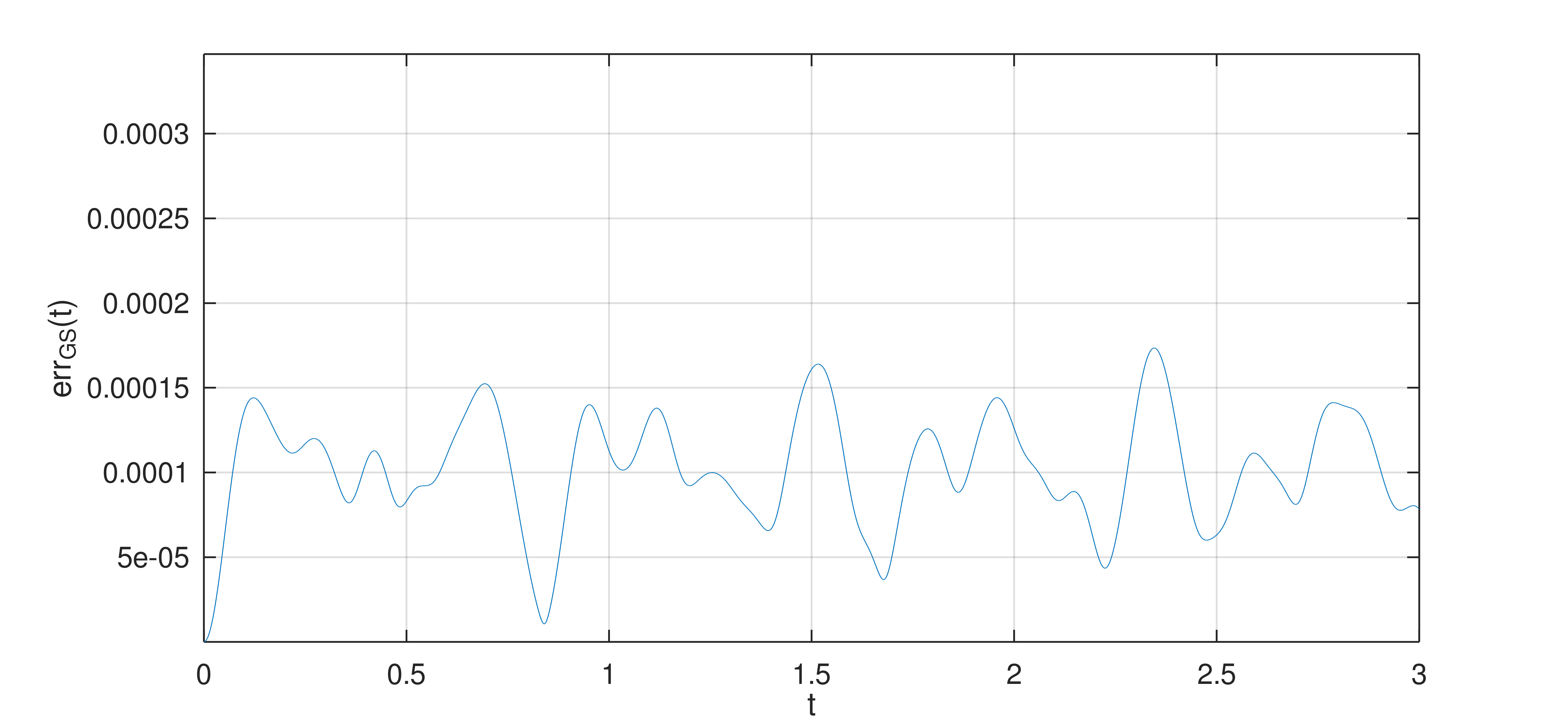}	
	\caption{Stability of the approximation of the ground state on $\left[0,T_{\max}\right]$.}
	\label{fig:evolution_ground_state.png}
\end{figure}

\subsection{Vortex nucleation}
We now compute the dynamics of the Gross-Pitaevskii equation \eqref{GP} with rotation and the set of parameters framed above, starting from the approximation of the ground state $U_{\GS}$ and adding this time the forcing potential $V_{\rot}$ defined by \eqref{forcing_potential_definition} with the set of parameters $ V_p=0.05$, $n_{\theta}=6$ and $\omega=10\pi/3$. We plot the square of the modulus of the solution vector $U(t)=\Phi_{\tau}^{t}(U_{\GS})$ for several times $t \in \left[0,T_{\max} \right]$. Throughout the dynamics, illustrated by Figure \ref{fig:short_time_dynamics.png}, we observe nucleation of vortices coming from the edge of the Bose-Einstein condensate. In particular, as time grows, more and more vortices do appear. Let's remark that no vortex appears if we perform the same simulations for the linear Schrödinger equation \eqref{GP} with $\gamma=0$. The same way, we do not observe vortex nucleation (at least over such period of time $\left[0,T_{\max} \right]$) throughout the dynamics for slower angular momentum such as $\omega=\pi$, as the velocity of the superfluid must pass the Landau criterion. 
\begin{figure}[h]
	\centering
		\includegraphics[width=0.9\textwidth,trim = 0cm 0cm 0cm 0cm, clip]{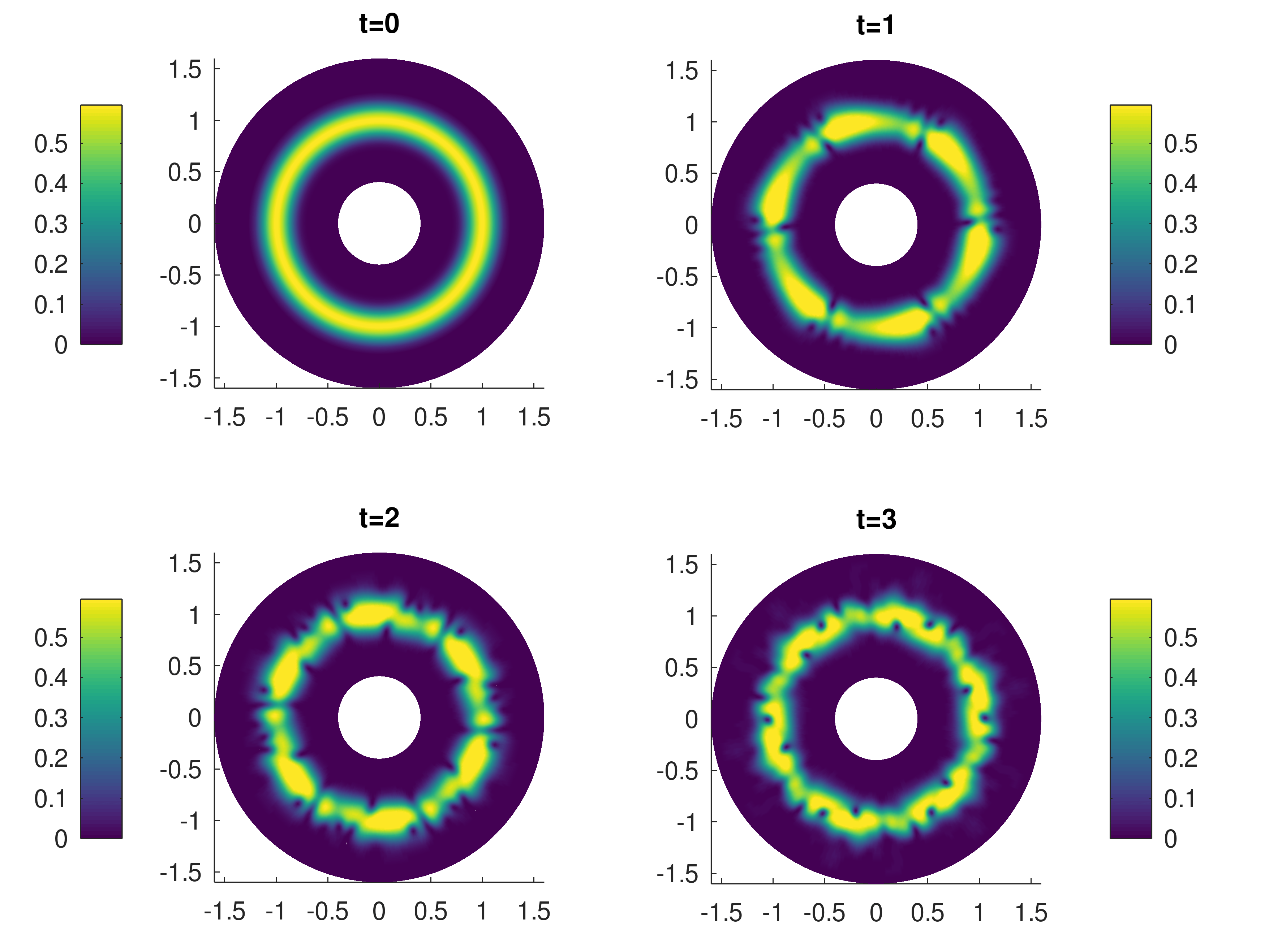}	
	\caption{The density $|U(t)|^2$ for times $t=0$, $1$, $2$ and $t=T_{\max}=3$.}
	\label{fig:short_time_dynamics.png}
\end{figure}

We now analyze the last frame (bottom right of Figure \ref{fig:short_time_dynamics.png}) at $t=T_{\max}=3$, plotting the results of the post-processing algorithm described in Section \ref{subsection_vortex_detection_post} with parameters $tol_1=0.1$, $tol_2=0.05$ and $\lambda_{\max}=10$, which numerically detects 12 vortices with 6 of indices 1 and 6 of indices -1 (we also display their respective characteristic length $\lambda_n$) which are plot in the first frame of Figure \ref{fig:post_processing_image.png}. We also plot the regularized vorticity from Section \ref{subsection_vortex_reg_vorticity} computed with regularization parameter $\delta=0.1$ in the second frame, as well as the pseudo-vorticity from Section \ref{subsection_vortex_pseudo_vorticity} with the same threshold in the third frame of Figure \ref{fig:post_processing_image.png}. Finally we plot the result of the vortex detection algorithm (with vortices of positive or negative indices displayed with different colors) based on the pseudo-vorticity fixing a threshold equal to $V_0/2$ at the fourth frame of Figure \ref{fig:post_processing_image.png}. 

One can make several comments from here. First, from a computation perspective, the first post-processing alogithm takes about 4630.7s on a common personal laptop computer, whereas the computation of the regularized vorticity and the vortex detection takes about 404.75s, and the pseudo-vorticity and detection about 117.94s. Then, these algorithms do not detect exactly the same number of vortices: the post-processing algorthim, which is based on the density function $|U(t)|^2$, has trouble detecting vortices at the edge of the condensate, whereas algorithm based on vorticity are able to track them more efficiently. It is worth noticing that these three algorthims highly depends on their respective parameters, and may require some tuning. Finally, we emphasize that vortices with negative indices nucleate from the inside of the ring, whereas vortices with positive indices, which are more numerous, nucleate from the outside (or the opposite if the potential $V$ is rotating the other way). They do not interact with each other, even in longer simulations, as negative vortices rotate at the inner edge of the condensate, and positive ones rotate at its outer edge.

\begin{figure}[h]
	\centering
		\includegraphics[width=0.9\textwidth,trim = 0cm 0cm 0cm 0cm, clip]{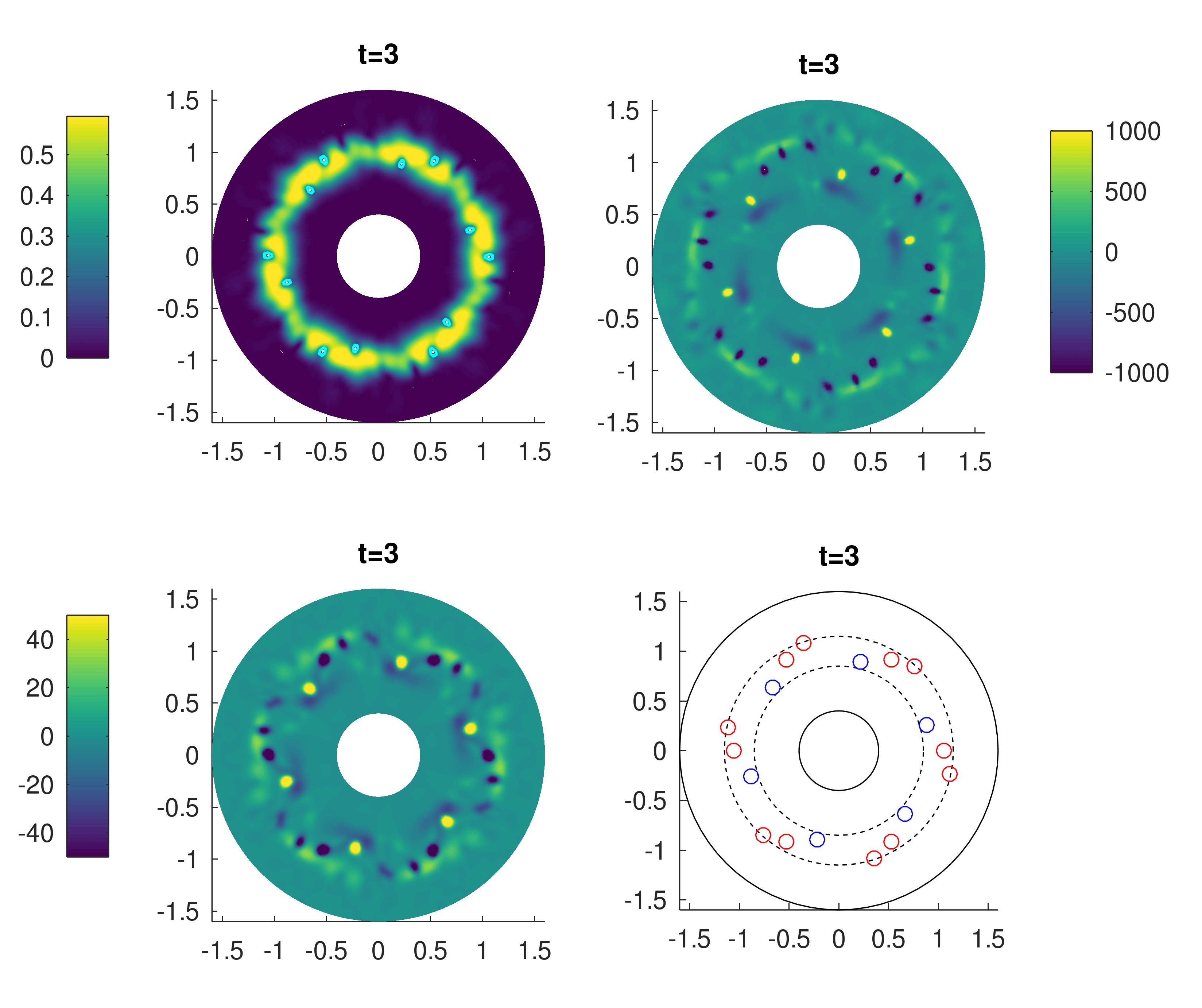}	
	\caption{\textit{(From left to right, top to bottom).} Post-processing algorithm, regularized vorticity, pseudo-vorticity and vortex detection algorithm from pseudo-vorticity of the wave function at time $t=T_{\max}$.}
	\label{fig:post_processing_image.png}
\end{figure}

We also plot in Figure \ref{fig:decomposition_2.png} both the density and the phase of the wave function based on the formula $\arg (\psi) = -i \log \left( \psi/|\psi|  \right)$ at initial time $t=0$ and final time $t=T_{\max}=3$, as well as the decomposition of the wave function into the basis of the linear operator as describe in Section \ref{subsection_decompo} for first radial modes $p=0,\ldots 3$ and with maximal azimuthal modes $L=80$. We observe that at initial state, the ground state is mainly localized on even modes, whereas during the dynamics, the nonlinearity tends to propagate the energy through higher modes, in particular in odd modes when vortices nucleate.

\begin{figure}[h]
	\centering
		\includegraphics[width=1\textwidth,trim = 0cm 0cm 0cm 0cm]{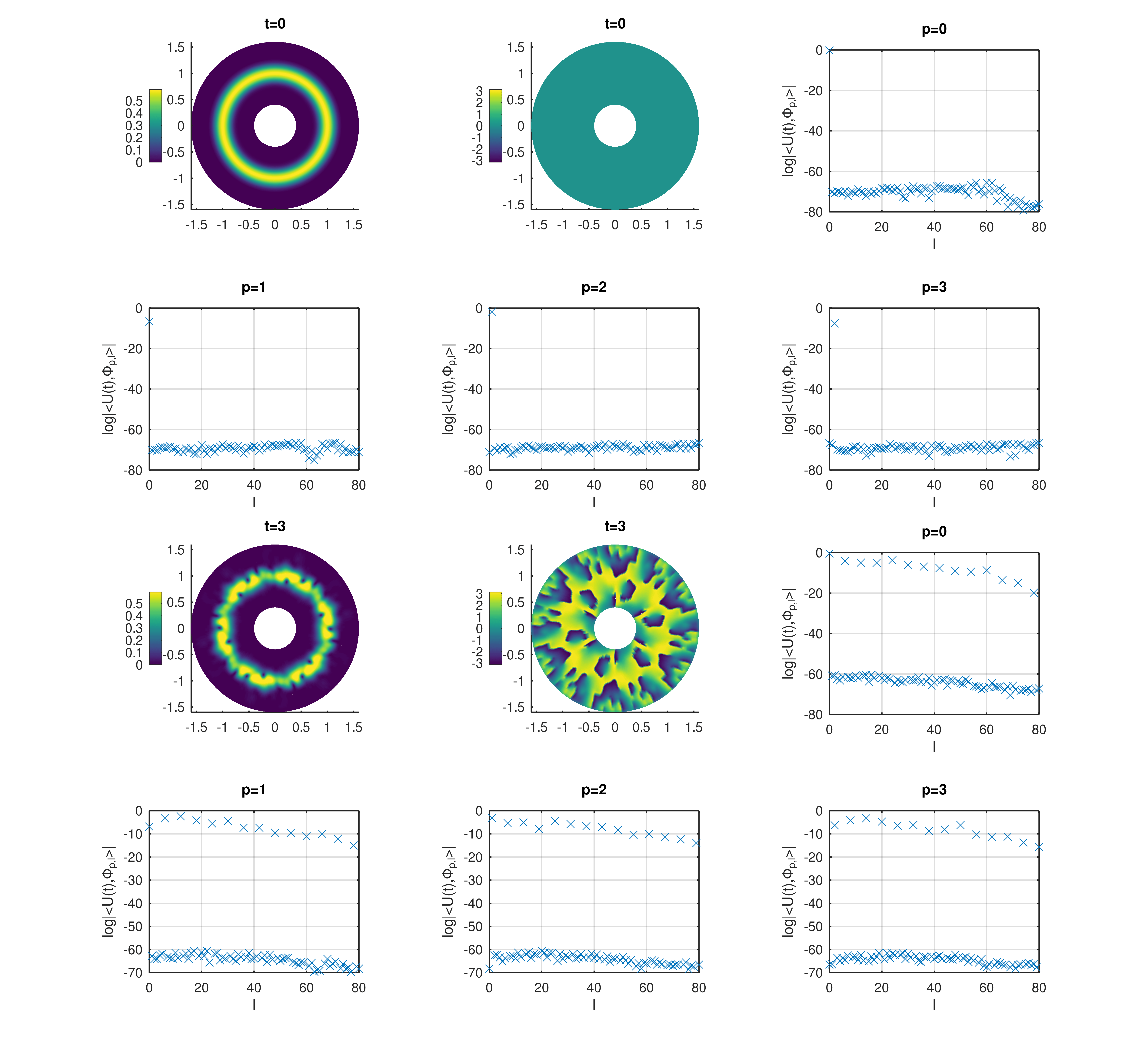}	
	\caption{\textit{(From left to right, top to bottom).} Density, phase and decomposition of the wave function on the first four radial modes (in logarithmic scale) at respectively time $t=0$ and time $t=T_{\max}$.}
	\label{fig:decomposition_2.png}
\end{figure}

\subsection{Other cases}

We also perform other numerical experiments on our ring geometry in order to highlight other known nonlinear phenomena in the context of quantum turbulence.

We now only consider the case without rotation, taking $V_{\rot}=0$, with the same parameters as above but with unstable initial condition. More precisely, we start our simulation with the initial state $\varphi(K)=U_{\GS}(K)$ if the triangle $K$ has a circumcenter $x_K$ with positive absciss $x_{K,1}>0$ , and $\varphi(K)=-U_{\GS}(K)$ otherwise (and we also regularize $\varphi$ by multiplying it with the function $x_K \mapsto \exp(-0.1/|x_{K,1}|)$). This creates an instability which implies the nucleation of 4 vortices at the outer edge of the condensate (see Figure \ref{fig:wave_sound.png}). When these vortices meets pairwise, they reconnect and vanish, which induces a dispersive wave throughout the condensate, also known as \textit{wave sound}. This wave emission is well known to play an important role for irreversible energy transfer mechanisms in quantum turbulent fluids \cite{proment2020}.

\begin{figure}[h]
	\centering
		\includegraphics[width=0.9\textwidth,trim = 0cm 0cm 0cm 0cm, clip]{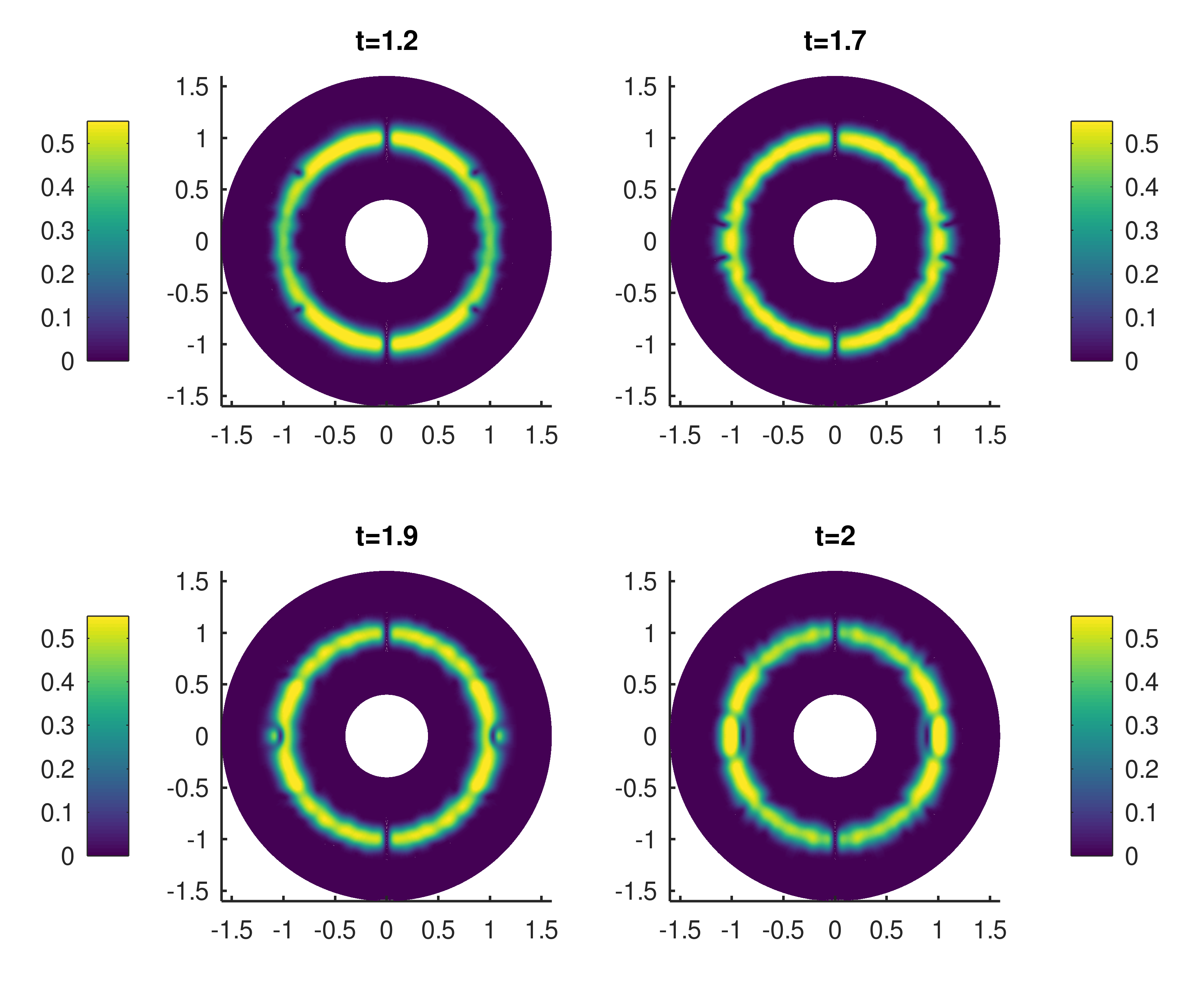}	
	\caption{\textit{(From left to right, top to bottom).} Density of the wave function for several times under transversal instability with Dirichlet boundary conditions at several times.}
	\label{fig:wave_sound.png}
\end{figure}

Finally, we mention that although all our theoretical results are based on Dirichlet boundary conditions, one can easily adapt our space discretization for Neumann boundary conditions (see for instance \cite{eymard2000}), which allows to perform numerical experiments where density does not vanish at the boundary of the ring domain. In Figure \ref{fig:Neumann.png} we reproduce the same numerical simulations as in Figure \ref{fig:wave_sound.png} described above, except that this time we take $V=0$ so that the ground state is a constant function over our geometry. The unstable initial state leads after some time to the nucleation of several vortices, which are briefly aligned in a snake-like transverse deformation which eventually breaks into vortex-antivortex pairs, a phenomenon which is highlighted in the work \cite{ricardo2010} for an harmonic trap. 

\begin{figure}[h]
	\centering
		\includegraphics[width=0.65\textwidth,trim = 0cm 0cm 0cm 0cm, clip]{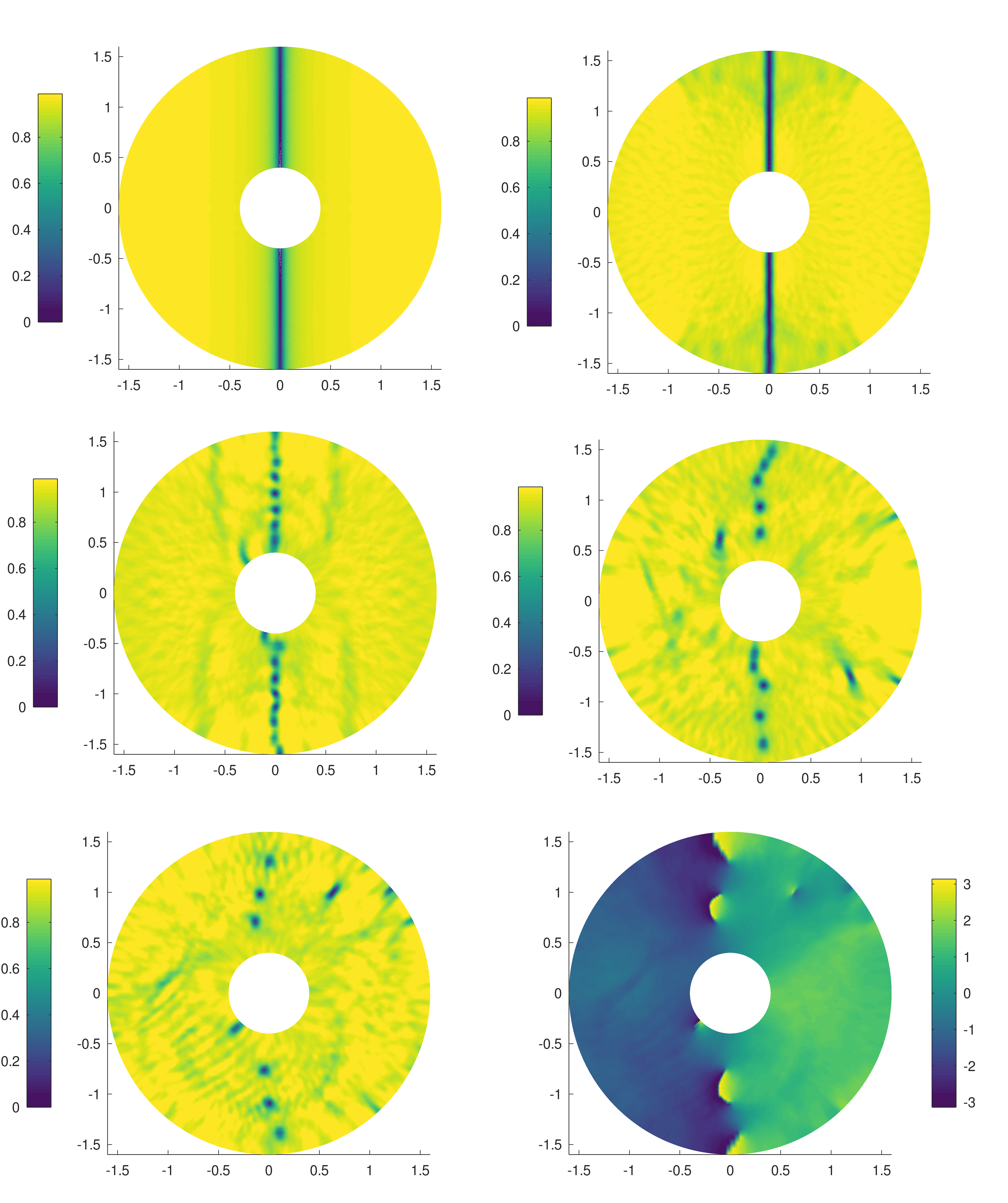}	
	\caption{\textit{(From left to right, top to bottom).} Density of the wave function for several times under transversal instability with Neumann boundary conditions at several times, and phase of the wave function at the last time.}
	\label{fig:Neumann.png}
\end{figure}

\subsection*{Acknowledgments} Q.C. and G.D. are supported by the Labex CEMPI (ANR-11-LABX-0007-01). The authors are grateful to Ricardo Carretero and Panayotis G. Kevrekidis for enlightening discussions about this work, as well as pointed out formulas for the regularized vorticity to efficiently track vortices (see Section \ref{subsection_vortex_reg_vorticity}). Q.C. is also thankful for numerous and valuable discussions at early stage of this work with participants of the workshop "Bridging Classical and Quantum Turbulence" which held at the Cargèse Scientific Institute (IESC) during July 2023, in particular to Francky Luddens.

\bibliographystyle{siam}
\bibliography{biblio}

\end{document}